\documentclass[10pt]{amsart}
 \usepackage{amsthm, tikz, amsfonts, amssymb, amsmath, amscd, etoolbox}
 \usetikzlibrary{lindenmayersystems,arrows.meta}
 \usepackage[eulergreek]{sansmath}
 \usepackage{pgfplots}
 \newcount\quadrant
\pgfdeclarelindenmayersystem{cayley}{
  \rule{A -> B [ R [A] [+A] [-A] ]}
  \symbol{R}{ \pgflsystemstep=0.5\pgflsystemstep } 
  \symbol{-}{
    \pgfmathsetcount\quadrant{Mod(\quadrant+1,4)}
    \tikzset{rotate=90}
  }
  \symbol{+}{
    \pgfmathsetcount\quadrant{Mod(\quadrant-1,4)}
    \tikzset{rotate=-90}
  }
  \symbol{B}{
    \draw [dot-cayley] (0,0) -- (\pgflsystemstep,0) 
       node [font=\footnotesize, midway, 
         anchor={270-mod(\the\quadrant,2)*90}, inner sep=.5ex] 
           {\ifcase\quadrant$a$\or$b$\or$a^{-1}$\or$b^{-1}$\fi};
    \tikzset{xshift=\pgflsystemstep}
  }
}
\tikzset{
  dot/.tip={Circle[sep=-1.5pt,length=3pt]}, cayley/.tip={Stealth[]dot[]}
}

\usepackage[auto]{contour}
\usepackage[all]{xy}
\usepackage{lmodern}
\usepackage{pslatex}
\makeatletter
\patchcmd{\@makechapterhead}
{\bfseries}
{\scshape}
\makeatother

\newcommand{\ve}{\varepsilon}

\newcommand{\C}{\mathbb{C}}
\newcommand{\R}{\mathbb{R}}
\newcommand{\Q}{\mathbb{Q}}
\newcommand{\Z}{\mathbb{Z}}

\newcommand{\TT}{\mathbb{T}}

\newcommand{\dol}{\bar{\partial}}

\newcommand{\h}{\hslash}
\newcommand{\D}{\mathcal{D}}

\newcommand{\restrace}{\mathrm{Res} \; \mathrm{Tr}}

\newcommand{\supp}{\mathrm{supp}}

\newcommand{\Res}{\textup{Res}}

\newcommand{\dom}{\textup{dom}}
\newcommand{\Bound}{\mathbb{B}}

\newcommand{\proj}{\mathrm{proj}}

\newcommand{\Schwartz}{\mathcal{S}}
\newcommand{\trace}{\textup{Trace}} 

\newcommand{\Hilb}{\mathcal{E}} 
\newcommand{\LL}{\mathcal{L}}

\newcommand{\T}{\TT}

\newcommand{\K}{\mathrm{K}}

\newcommand{\KK}{\mathrm{KK}}

\newcommand{\PD}{\textup{PD}}
\newcommand{\pnt}{\textup{pt}}

\newcommand{\ev}{\textup{ev}}

\newcommand{\Poincare}{\mathcal{P}}

\newcommand{\Tr}{\mathrm{Tr}}

\DeclareMathOperator{\cosech}{csch}

\numberwithin{equation}{section}
\newcommand*{\norm}[1]{\lVert#1\rVert}
\newcommand*{\abs}[1]{\lvert#1\rvert}

\newtheorem{theorem}{Theorem}[section]

\newtheorem{lemma}[theorem]{Lemma}
\newtheorem{corollary}[theorem]{Corollary}
\newtheorem{proposition}[theorem]{Proposition}

\theoremstyle{definition}
\newtheorem{definition}[theorem]{Definition}

\newtheorem{remark}[theorem]{Remark}

\begin{document}
\title{Zeta functions and topology of Heisenberg cycles for linear ergodic flows }


 \author{Nathaniel Butler}
\author{Heath Emerson}
\email{hemerson@math.uvic.ca}
\address{Department of Mathematics and Statistics\\
  University of Victoria\\
  PO BOX 3045 STN CSC\\
  Victoria, B.C.\\
  Canada V8W 3P4}
  
 \author{Tyler Schulz}


\date{\today}

\thanks{This research was supported by an NSERC Discovery grant and the NSERC USRA program}

\begin{abstract}
Placing a Dirac-Schr\"odinger operator along the orbit of a flow on a compact 
manifold \(M\) defines an \(\R\)-equivariant spectral triple over the algebra of smooth 
functions on \(M\). We study some of the properties of these triples, especially 
their zeta functions, which have the form \(\trace (fH^{-s})\) with \(f\) the restriction to 
\(\R\) of a function on \(M\) and \(H = -\frac{\partial^2}{\partial x^2} + x^2\) the harmonic 
oscillator. The meromorphic continuation property and pole structure of these 
zeta functions is related to ergodic time averages in dynamics. The construction 
reproduces the `Heisenberg cycles' of Lesch and Moscovici, in the case of 
the periodic flow on the circle, where it produces a spectral triple over the smooth irrational torus
in the irrational rotation algebra \(A_\h\). 
 We strengthen a result of these authors, showing that 
the zeta function \(\trace (aH^{-s})\) extend mermomorphically  for any element \(a\) of the 
C*-algebra \(A_\h\). Another variant of the construction produces a spectral cycle for 
\(A_\h\otimes A_{1/\h}\) and a spectral triple over a suitable subalgebra with the meromorphic 
continuation property if \(\h\) satisfies a Diophantine condition. The class of this cycle defines a 
fundamental class in the sense that it determines a KK-duality. We employ the Local Index Theorem of Connes and Moscovici in order to elaborate
an index theorem of Connes for certain classes of differential operators on the line and 
compute the intersection form on K-theory induced by the fundamental class.

\end{abstract}

\maketitle


\section{Introduction}

The irrational rotation algebra \(A_\h := C(\TT)\rtimes_\h \Z\), the crossed product 
of \(C(\TT) = C(\R/ \Z)\) by the action of \(\Z\) by translation by \(\h \in \R\setminus \Q\) mod \(\Z\) on 
\(\TT\), is one of the key motivating examples in Noncommutative Geometry. 
Early results of Connes and Rieffel classified finitely generated projective modules 
over \(A_\h\), or over its natural Schwartz subalgebra \(A_\h^\infty\), by 
an analogue of the first Chern number of a line bundle over \(\TT^2\), defined for 
\(e\in A_\h^\infty \subset A_\h\), by
\[ c_1(e) := \frac{1}{2\pi i} \cdot \tau (e[\delta_1 (e), \delta_2(e)]),\]
where \(\delta_1, \delta_2\) are the derivations of \(A_\h\) generating the 
natural \(\R^2\)-action, and \(\tau\) is the trace. In fact these numbers are \emph{integers}, a fact 
related to the Quantum Hall effect in solid state physics. 

The reason for the integrality lies in the following. 
The 
densely defined 
operators \(\frac{\partial}{\partial x}, \frac{\partial}{\partial y}\) 
on \(L^2(\TT^2)\) assemble to the operator 
\[ \dol := \begin{bmatrix} 0 & \frac{\partial}{\partial x} - i \frac{\partial}{\partial y}\\
\frac{\partial}{\partial x} + i \frac{\partial}{\partial y} & 0 \end{bmatrix}.\]
on \(L^2(\TT^2)\oplus L^2(\TT^2)\), and the representation of 
\(C(\TT^2)\) on \(L^2(\TT^2)\) by multiplication operators 
can be adjusted by introducing phase factors to 
give a representation \(\lambda_\h \colon A_\h \to \Bound\left( L^2(\TT^2)\right)\) which 
makes the triple 
\( \left( L^2(\TT^2)\oplus L^2(\TT^2), \lambda_\h, \dol\right)\) a \(2\)-summable 
spectral triple over 
\(A_\h^\infty\) whose Chern character may be computed using the Local Index 
Formula of Connes and Moscovici 
to be the class of the cyclic cocycle 

\begin{equation}
\label{equation:sldkfjlskdjfweiur2}
 \tau_2 (a^0,a^1, a^2) = \tau \left( a^0\delta_1(a^1)\delta_2(a^2) -
  a^0\delta_2(a^1)\delta_1(a^2)\right), \;\;\;\; a^0, a^1, a^2\in A_\h.
 \end{equation}
The integrality of the Chern numbers 
\(\tau_2(e,e,e)\) follows from the Connes-Moscovici Index Theorem which implies that 
for any idempotent \(e\in A^\infty_\h\), 
\[  c_1(e) = \tau_2(e,e,e)  = \langle [e], [\dol]\rangle \in \Z,\]
where the right hand side is the pairing between K-theory and K-homology. But 
it is a result going back to early direct computations of Connes \cite{Connes} involving 
in particular a calculation of the cyclic cohomology of \(A_\h^\infty\).

In this article, we study a slightly different method of constructing spectral triples, 
using the  
 operators \(x \pm d/dx\), the annihilation and creation operators of 
 quantum mechanics. They assemble to form a spectral triple over a suitable 
smooth subalgebra of \(C_u(\R)\rtimes \R_d\), where \(C_u(\R)\) is the C*-algebra of 
uniformly continuous, bounded functions on \(\R\) and \(\R_d\) is the group of real numbers 
with the discrete topology. The operator of the triple is 
\( D = \begin{bmatrix} 0 & x-d/dx\\ x+ d/dx& 0 \end{bmatrix}\), 
whose closure is self-adjoint. The operator \(D\) 
commutes 
mod bounded operators with group translations and smooth bounded 
functions on \(\R\) with bounded derivatives, and \( D^2\) is essentially the
 direct sum of two copies of the harmonic 
oscillator
\[ H = -\frac{d^2}{dx^2} + x^2 \]
 on \(\R\), which has discrete spectrum consisting of the odd positive integers. 
 The representation \(\pi\) of \(C_u(\R)\rtimes \R_d\) on \(L^2(\R)\) 
 lets 
 \(f\in C_u(\R)\) act by the corresponding multiplication operator \( (f\xi)(x) = f(x) \xi (x)\), 
 and a group element \(t\in \R_d\) by the group translation 
 unitary operator \((u_t\xi )(x) = \xi (x-t)\).

The triple just describes gives a spectral (unbounded) cycle for 
 \(\KK_0(C_u(\R)\rtimes \R_d, \C)\). We call it the \emph{Heisenberg cycle}. 
 
 The Heisenberg cycle pulls back to any C*-subalgebra of \(C_u(\R)\rtimes \R_d\), 
 and in this article we are most interested in subalgebras arising from ergodic flows. 
 If \(\alpha\) is a smooth flow on a compact manifold \(M\), 
 If \(p\in M\) the function 
\( f_p(t) := f \left(\alpha_t (p)\right)\)
is uniformly continuous on \(\R\) if \(f\)is continuous. It follows that restriction to an 
orbit defines an embedding \(B_\alpha  := 
C(M)\rtimes_\alpha \R_d\subset C_u(\R)\rtimes \R_d\).
 So one can associate a Heisenberg cycle to any 
smooth flow and corresponding class \([B_\alpha]\in \KK_0(C(M)\rtimes \R_d, \C)\) 
or, in \(\KK_0(C(M)\rtimes \Lambda, \C)\), if one has a 
subgroup \(\Lambda\subset \R_d\) of particular interest for the context. 

If one takes the 
trivial subgroup, then the class in 
\(\KK_0(C(M), \C)\) is equal to the class in K-homology of the point \(p\in M\), and 
so contains no interesting topological information (this follows from constructing a certain 
homotopy in KK, see \cite{DE}).
 However, simple examples show that 
for certain natural (nontrivial) choices of subgroup, one obtains a great deal of 
topological information about the crossed-products.

In the case of the periodic flow on \(\T\) and \(\Lambda = \Z \h\), where \(\h\) is 
irrational, the C*-algebra \(B_\alpha\) is \(C(\T)\rtimes \R_d \) which contains the 
irrational rotation algebra 
\(A_{\h} = C(\T)\rtimes \h \Z\) by restricting to the subgroup \(\h \Z\subset \R_d\). This is 
the irrational rotation algebra \(A_\h\). 
The Heisenberg cycle for this algebra, has been studied by Connes \cite{Connes:note}, \cite{Connes}, 
and Moscovici and Lesch \cite{LM}. 
The latter authors refer to \emph{Heisenberg modules}. One can build a Heisenberg 
module by twisting 
the Dirac-Dolbeault cycle of Connes by a Morita bimodule; such bimodules come from 
compact transversals to the Kr\"onecker flow. One obtains thus a family 
of such cycles (for \(A_\h\)) all having a somewhat similar form, and involving the Dirac-Schr\"odinger operators
\(x\pm d/dx\) on \(L^2(\R)\), or a finite sum of copies of \(L^2(\R)\). 

Our Heisenberg cycles are defined over a much larger algebra \(C_u(\R)\rtimes \R_d\) than 
\(A_\h\). One gets cycles for \(C(M)\rtimes \Lambda\) for flows on manifolds \(M\), and 
so in principal can be used to find topological invariants of flows. 
We restrict ourselves in this note to looking at Kr\"onecker flow. 
For the subgroup \(\Lambda \subset \R_d\) generated by \(1, \h\), the 
crossed-product \(B_\h := C(\T^2)\rtimes \Lambda\) is isomorphic to \(A_\h \otimes A_{1/\h}\). 
In the second part of the paper we compute some topological invariants of these 
Heisenberg cycles, using the Local Index Theorem of Connes and Moscovici and 
N. Higson's exposition of it in \cite{Higson:Local}. 
The Heisenberg cycle for 
\(A_\h \otimes A_{1/\h} =B_\h\) induces a KK-duality between \(A_\h \) and 
\(A_{1/\h}\) and we 
compute the index pairing \(\K_0(A_\h)\times \K_0(A_{1/\h}) \to \Z\) and 
show that it has matrix \[\begin{bmatrix} 1 & -\left \lfloor{1/\h}\right \rfloor  \\ -\left \lfloor{\h}\right \rfloor  
& 1\end{bmatrix} \]with respect to the bases consisting of the unit and the Rieffel projections. 
This strengthens an index calculation of Connes in \cite{Connes} for classes of 
differential operators on the real line. This is based on our computation of 
the Chern character of the Heisenberg cycle over \(A_\h\), which we show is given by the 
mixed degree cyclic cochain
\[ \tau - \h \tau_2,\]
where \(\tau_2\) is as in \eqref{equation:sldkfjlskdjfweiur2}.

The main technical contribution of this note concerns the meromorphic extension 
problem of the zeta functions \[\zeta( a, s) := \trace (aH^{-s})\] for 
 \(a\in C_u(\R)\rtimes \R_d\). Establishing such meromorphic extensions is necessary to 
 apply the Local Index Theorem, at least in the presentation \cite{Higson:Local}, as the
cyclic cocycles involved in the local Chern character formula
 are obtained as poles of such zeta functions.

 The meromorphic extension property 
 in the classical situation, asserts that if \(\Delta_M\) is the Laplacian on a compact 
 manifold, and \(f\in C^\infty (M)\), then \(\trace (f\Delta_M^{-s})\) extends 
 meromorphically to \(\C\), with certain poles; it is proved by the theory of 
 asymptotic expansions, specifically of the kernel of \(f e^{-t\Delta_M}\), because the
 Mellin transform transforms the meromorphic extension problem into a problem 
 about the asymptotics of the heat kernel as \(t\to 0\). Such asymptotic expansions 
 are also available for the situation of the Schwartz algebra of the 
 irrational rotation algebra \(A_\h\), as 
 noted by \cite{LM}, who used them to deduce the meromorphic extendibility of \(\zeta (aH^{-s})\) 
 for \(a\in A_\h^\infty\) in the \emph{smooth} irrational torus. In fact we show that 
 \(\zeta(a,s)\) meromorphically extends for \(a\) in the \emph{C*-algebra} \(A_\h\), and 
 that more generally, the zeta functions for flows appear to be related
  to ergodic time 
averages in dynamics. If \(\alpha\) is a smooth ergodic flow on \(M\) then 
we show that 
\[ \lim_{s\to 1^+} (s-1)\cdot  \trace (f_p H^{-s}) = \int_M fd\mu\]
for a.e. \(p\in M\), \(f\in C(M)\), and \(\mu\) any \(\alpha\)-invariant measure. 
Therefore, the residue trace, defined spectrally, recovers the invariant  measure \(\mu\). 
The proof is based on an integral formula for \(\trace (fH^{-s})\); 
in fact, as we show more generally that if \(f\in C_u(\R)\) and if 
\(\lim_{T\to \infty} \frac{1}{T}\int_0^T f(t)dt\) exists, then the limit equals 
\(\lim_{s\to 1^+} (s-1)\cdot  \trace (f_p H^{-s}) \) of \(f_p\). The result then
 follows from
 the Birkhoff Ergodic Theorem. 

The meromorphic extension property from this point of view, for a given smooth flow, 
requires 
a strengthening of the 
Birkhoff Ergodic theorem for that situation which gives a finer estimate for the deviation 
\( \int_0^T f(t)d\mu -  T\int_M fd\mu.\)
 From our integral formula for the zeta function, it is apparent that the 
  meromorphic extension property of \(\trace (f_pH^{-s})\) would follow, for example, 
 for any \(f\in C^\infty (M)\), if one was guaranteed 
smooth solvability of the cohomological equation \(Xu = f\) for a smooth 
flow with generating vector field \(X\). 
The condition \(\int_M f d\mu = 0\) of \(f\) is an obvious obstruction 
to \(Xu = f\) being continuously solvable for any \(\alpha\)-invariant \(\mu\). 
For the standard periodic flow on the 
circle, this is the only obstruction. This is because if \(f\) is continuous 
and \(\rho\)-periodic and \(\int_0^\rho  fd\mu  = 0\) then the anti-derivative 
\(F(T) := \int_0^T f(t) dt\) is also \(\rho\)-periodic, so \(F\) solves the equation continuously. 
As we show, then 
the zeta function 
can be meromorphically extended to \(\mathrm{Re} (s ) > 1-\frac{n}{2}\) 
by solving the equation \(n\) times, and so meromorphically 
extended to \(\C\). For the Kr\"onecker flow on 
\(\T^2\), the cohomological equation
 \(Xu=f\) is smoothly solvable 
for smooth \(f\) of zero Lebesgue mean if \(\alpha\) satisfies a Diophantine condition, 
and it follows that \(\trace (fH^{-s})\) extends meromorphically to \(\C\) with a 
simple pole at \(s=1\) in this case as well, if \(f\) is \emph{smooth}. As observed in 
\cite{Forni1} a more refined 
statement is possible. The passage from \(f\) to \(u\) in solving the cohomological 
equation involves a specific 
loss of Sobolev regularity related to the Diophantine constant. Hence 
if \(f\) lies in a sufficiently high Sobolev space for \(\T^2\), then 
\(\trace (f_pH^{-s})\) can be meromorphically extended a certain finite distance. 

The issue of estimating 
deviations from ergodic averages is a research topic of significant activity
 (see \cite{Forni1}, \cite{Forni2}). It would be interesting to see if the meromorphic 
 extension property for the Heisenberg cycles holds for more general flows. 
 
\section{Spectral cycles from the canonical anti-commutation relations}

The Heisenberg group  
\( H = \{ \begin{bmatrix} 1 & x & z \\ 0 & 1 & y\\ 0 & 0 & 1\end{bmatrix} \; | \; x, y, z \in \R\}\)
has Lie algebra \(\mathfrak{h}\)  
the \(3\)-by-\(3\) strictly upper triangular matrices under matrix commutator. Let 
\(X, Y\) be the elements 
\[ X = \begin{bmatrix} 0 &1 & 0 \\ 0 & 0 & 0\\ 0 & 0 & 0\end{bmatrix}, \;\;\;\;
Y =  \begin{bmatrix} 0 &0 & 0 \\ 0 & 0 & 1\\ 0 & 0 & 0\end{bmatrix},\] 
of \(\mathfrak{h}\). Then 
\[ [X, Y] =   Z := \begin{bmatrix} 0 &0 & 1 \\ 0 & 0 & 0\\ 0 & 0 & 0\end{bmatrix},\]
while \(Z\) is central in \(\mathfrak{h}\).  
It follows that if 
\(\pi \) is any \emph{irreducible} representation of \(H\), 
\(\pi (Z) = \pi ([X,Y]) = [\pi (X), \pi (Y)] \) is a multiple of the identity operator: 
\[ [\pi (X), \pi (Y)] = \hslash,\]
for some \(\hslash \in \R\), a `Planck constant.' 

The name \emph{Heisenberg group} originates in these
 relations, which have the same form as the canonical commutation relations in quantum mechanics,
 where \(x\) and \(\frac{d}{dx}\) model position and momentum operators. 

From the above remarks, we obtain
 a classification of irreducible 
representations of \(H\). Either \(\hslash = 0\), in which case \(\pi (Z) = 0\) and hence 
\( \pi (X)\) and \(\pi (Y)\) commute, which implies the representation is \(1\)-dimensional, 
and is completely determined by the pair of real numbers \( (\pi (X), \pi (Y))\), or 
\(\hslash \not= 0\), in which 
case one can show that 
the representation is isomorphic to the following interesting representation 
\(\pi_\h\) of \(\mathfrak{h}\) by unbounded operators on \(L^2(\R)\). 
Let
\[ \pi_\h (X) = x, \;\; \textup{and} \;\;\;\;\pi_\h (Y) = \h \frac{d}{dx}.\]
Then 
\( [x, \h  \frac{d}{dx}] = \h\), so the required identity is satisfied
 to give a representation.

Application of functional calculus to the operators \(x\) and 
\(\frac{d}{dx}\) produces the operators 
\[ u = e^{2\pi i x}, \;\;\;\; v_\h := e^{-\h \frac{d}{dx}},\]
where \(u\) is multiplication by the periodic function \(e^{2\pi i x}\) and 
\[ (v_\h)\xi (x) = \xi (x-\h).\]
We have 
\[ uv_\lambda = e^{-2\pi i \h} v_\lambda u.\]
If \(\h \in \R\setminus \Q\) then the
 \emph{irrational rotation algebra} is the C*-algebra 
\[A_\h := C(\TT)\rtimes_\h \Z,\]
where \(\Z\) acts on the circle \(\TT:= \R/\Z\) 
with generator the automorphism induced by 
translation by \(\h\; \textup{mod} \; \Z\). 
If \(U \in C(\TT)\rtimes_h \Z\) is the generator \(U(t) = e^{2\pi i t}\) of \(C(\TT)\) and 
\(V\) the generator of the \(\Z\) action in the crossed-product, then a quick computation 
shows that 
\[UV= e^{-2\pi i \h} VU \; \in A_\h,\]
and it follows that we obtain, for each \(\h \), rational or not,  
a representation
\[ \pi_\h \colon A_\h \to \Bound (L^2(\R))\]
of \(A_\h\) on \(L^2(\R)\). Note that \(\pi_\h\) depends on \(\h\) as a real number, while 
\(A_\h\) only depends on the class of \(\h\) mod \(\Z\).

We are going to fit these representations into a spectral 
cycle for \(\KK_0(A_\h ,\C)\), using the properties of the \emph{harmonic oscillator} 
\begin{equation}\label{harmonicoscillator}
H := -\frac{d^2}{dx^2} + x^2,\end{equation}
a second-order elliptic operator on \(\R\), whose domain we will take initially to 
be the Schwartz space \(\Schwartz (\R)\). Actually, the construction is more general, and 
produces a spectral cycle for \(C_u(\R)\rtimes \R_d\), with \(\R_d\) denoting 
\(\R\) with the discrete topology.

Let  \(A  = x + \frac{d}{dx}\), initial domain the Schwartz 
space \(\Schwartz (\R)\), and 
\(A^* = x-\frac{d}{dx}\). The relations 

 \begin{equation}
\label{equation:provingbottandquantummechanicsfgddgf8}
AA^* = H + 1, \;\; A^*A = H-1,\;\; [A,A^*] = 2,\;\;\; [H, A] = -2A, \;\;[H, A^*] = 2A^*.
 \end{equation}
hold as operators on \(\Schwartz\). (See \cite{Roe}.)

Now set 
\( \psi_0 := \pi^{-\frac{1}{4}} \cdot e^{-\frac{x^2}{2}} \in L^2(\R).\)
In quantum mechanics, \(\psi_0\) is called the \emph{ground state}, and the states 
inductively defined 
by 
\( \psi_k :=   (2k)^{-\frac{1}{2}} \cdot A^*\psi_{k-1} \)
the `excited states'.  Observe that due to 
\( HA^* = A^*H + 2A^*,\)
from \eqref{equation:provingbottandquantummechanicsfgddgf8}, we see
by induction that \(\psi_k\) is a unit-length
 eigenvector of \(H\) with eigenvalue \(2k+1\):  
\begin{multline}
 H\psi_k = 
 (2k)^{-\frac{1}{2}} \cdot HA^*\psi_{k-1} 
 =  (2k)^{-\frac{1}{2}} \cdot ( A^*H + 2A^*) \xi_{k-1} 
\\ =  (2k)^{-\frac{1}{2}} \cdot (  (2k-1) \cdot  A^* \psi_{k-1} + 2A^*\psi_{k-1})  
 = (2k+1) \cdot \psi_k.
 \end{multline}

It follows from \( [H, A] = -2A\) that 
 \[ A\psi_k = \sqrt{2k}\cdot \psi_{k-1},\;\;\;\; A^*\psi_k = \sqrt{2k+2} \cdot \psi_{k+1}.\]

The eigenvectors of \(H\) are given by 
\( \xi_k = H_k (x)e^{-\frac{x^2}{2}}\)
where \(H_k\) is the \(k\)th \emph{Hermite polynomial}. 
This follows from induction using the recurrence 
\[ H_k (x) =  (2k)^{-\frac{1}{2}} \cdot \bigl( 2xH_{k-1} (x) - H_{k-1}' (x) \bigr)\]
to define the polynomials.

 The vectors \( \{ \psi_k\}\) form an orthonormal basis for \(L^2(\R)\) by 
 the Stone-Weierstrass Theorem, and each 
 \(\psi_k\) is in the Schwartz class \(\Schwartz (\R)\). 
 
  With respect to this basis, \(H\) 
 is diagonal with eigenvalues the odd integers \(1 , 3, 5, \ldots\): 
 \[ H  = \begin{bmatrix} 1 & 0 & 0 & \cdots \\
0 & 3 & 0  & \cdots \\
0 & 0  & 5 &\cdots\\
\cdots & \cdots & \cdots & \cdots\end{bmatrix}.\]

 In particular, \(H\) has a canonical extension to a self-adjoint operator on 
 \(L^2(\R\)), and \(f (H)\) is a compact operator for all \(f\in C_0(\R)\), and a bounded 
 operator for all \(f\in C_b(\R)\).

If \(f\in L^2(\R)\), let \( (\hat{f} (n))\) denote the sequence of its 
Fourier coefficients with respect to the spectral decomposition of 
 \(L^2(\R)\) into eigenspaces of \(H\) discussed above. 

\begin{lemma}
If \(f\in L^2(\R)\), then 
\(f\in \Schwartz\) if and only if \( (\hat{f}(n))\) is a rapidly decreasing sequence of integers: 
\[ \abs{\hat{f} (n)} = O(n^{-k})\]
for any \(k\). 
\end{lemma}

The proof is routine, see \cite{Roe}.

Let \(D\) be the unbounded operator 
\[ D = \begin{bmatrix} 0 & A^* \\ A & 0\end{bmatrix}\]
on \(L^2(\R) \oplus L^2(\R)\), defined initially on Schwartz functions; it admits a 
canonical extension to a densely defined self-adjoint operator on \(L^2(\R)\). Since
\(D^2 = \begin{bmatrix} H-1 & 0 \\ 0 & H+1\end{bmatrix}\), 
 \( 1+D^2 = \begin{bmatrix} H & 0 \\ 0 & H+2\end{bmatrix}\), which is 
now diagonal with respect to the basis described above, and 
invertible as an unbounded operator.

If \(f\in C_b^\infty(\R)\) is a smooth 
bounded function with bounded first derivative, acting by 
a multiplication operator on \(L^2(\R)\), then the commutator 
\([f, D] = \begin{bmatrix} 0 & -f' \\ f' & 0 \end{bmatrix}\)
is a bounded operator. Let 
\[ C_u (\R) := \{ f\in C_b(\R) \; | \;   f \;  \textup{is uniformly continuous} \}\]
be the C*-algebra of bounded uniformly continuous functions on \(\R\). 
The group \(\R_d\) of real numbers \emph{with the discrete topology}, acts on 
\(C_u(\R)\). Let \(\pi \colon C_u (\R)\rtimes \R_d \to \Bound (L^2(\R))\) 
the representation of \(C_u(\R)\) by multiplication operators and 
\(\R\) by translations.

\begin{proposition}
\label{universalspectrual}
The 
triple 
\[ \left(L^2(\R) \oplus L^2(\R), \pi \oplus \pi,\;\; D = \begin{bmatrix} 0 & A^* \\ A & 0\end{bmatrix}\right)\]
is a spectral triple over \(C_u^\infty (\R)[\R^d] \subset 
C_u(\R)\rtimes \R_d\); it is \(2\)-dimensional in the sense that 
\( \abs{D}^{-2} \in \mathcal{L}^{(1,\infty)}\).

\end{proposition}

We refer to the cycle above as the \emph{Heisenberg cycle}.

As \(C_u^\infty(\R)[\R_d]\) is dense in \(C_u(\R)\rtimes \R_d\), the
 Proposition implies that the associated Fredholm module 
\begin{equation}
\label{eq:heis}
 \left(L^2(\R) \oplus L^2(\R), \pi  \oplus \pi,\;\; F:= \chi (D)  
= \begin{bmatrix} 0 & A^*(H+2)^{-\frac{1}{2}} \\ AH^{-\frac{1}{2}} & 0 \end{bmatrix}\right).\end{equation}
obtained by applying a normalizing function \(\chi\), here chosen to be
\(\chi (x) = x(1+x^2)^{-\frac{1}{2}}\), 
defines a cycle for \(\KK_0(C_u(\R)\rtimes \R_d, \C)\), because 
\([\pi (a) , AH^{-\frac{1}{2}}]\) is 
a compact operator for \(a\in C_u (\R)\rtimes\R_d\) by standard functional calculus arguments (see \cite{Kasparov}, \cite{HR}) due to \([\pi (a), A]\) being bounded for dense \(a\).

The corresponding class in \(\KK_0 (C_u(\R)\rtimes \R_d, \C)\) is non-zero: it has index \(+1\).

Since the spectrum of \(H\) grows linearly, \(\pi (a) H^{-s}\) is trace-class for \(\mathrm{Re} (s)>1\) and 
 the zeta function \(\trace (\pi (a) H^{-s})\) is holomorphic for \(\mathrm{Re} (s) >1\) and 
\(a\in C_u(\R)\rtimes\R_d\). One of our main interests is in the possible meromorphic continuation properties 
of such zeta functions. 

The irrational rotation algebra \(A_\h\) is a subalgebra of \(C_u(\R)\rtimes\R_d\) and 
the restriction of the representation \(\pi\) above to \(A_\h\) lets \(f\in C(\TT) = C(\R/\Z)\) 
act by multiplication on \(L^2(\R)\) by the corresponding periodic function, 
and the group \(\Z\) by \(n \mapsto u_{n\h}\), with, recall, 
\( u_t \xi  (x) = \xi (x-t).\) However \(C_u(\R)\rtimes \R_d\) contains numerous other 
subalgebras of related interest. We first point out a generic example of such a subalgebra, arising 
from dynamics.

\begin{lemma}
\label{pullbacktoflow}
If \(M\) is a compact manifold and \(\{\alpha_t\}_{t\in \R}\) is a smooth flow on 
\(M\), then if \(p\in M\), then mapping \(f\in C(M)\) to the 
uniformly continous function \(f_p(t) := f(\alpha_t p)\) on \(\R\), 
 and mapping \(t\in \R_d\) to \( u_t\), determine
  a C*-algebra homomorphism 
\[ \mu \colon C(M)\rtimes_\alpha \R_d\to C_u(\R)\rtimes \R_d\]
where \(\R_d\) is the group of real numbers with the discrete topology. 
It is injective if the flow is minimal, and restricts to a *-algebra homomorphism 
\(C^\infty (M)[\R_d] \to C^\infty_u (\R)[\R_d]\). 
\end{lemma}

The proof is the observation that if the vector field \(X\) generates the flow, then 
 \(f\in C^\infty (M)\) implies 
\(\frac{d}{dt} f_p(t) = X(f)_p(t) \) which is bounded in \(t\) so 
\(f_p\) is uniformly continuous on \(\R\). 
In particular the Heisenberg cycle pulls back to a cycle for \(C(M)\rtimes \R_d\), and 
a spectral triple over \(C^\infty (M)[\R_d]\).

Returning to irrational rotation, fix \(\h \in \R\), so that we have the representation 
\(\pi_\h\) 
of \(A_\h\) determined by \(C(\TT) = C(\R/\Z) \) acting by multiplication operators by 
\(\Z\)-periodic functions, 
\(n \in \Z\) by \(u_{n\h}\).

\begin{lemma}
\label{whatismu}
Let \(\pi^\h \colon A_{1/\h} \to \Bound (L^2(\R))\) be the representation 
obtained by letting \(f\in C(\TT) = C(\R/\Z)\) act by multiplication by 
\( f( \frac{ x}{\h})\) and \(n\in \Z\) by translation by \(n\). 

Then \(\pi_\h(A_\h)\) and \(\pi^\h (A_{ 1/\h})\) commute. The tensor product 
\(\rho_\h(a\otimes b) := \pi_\h (a)\pi^\h (b)\) of the representations 
gives an representation of \(B_\h\) on \(L^2(\R)\) factoring through 
\[\mu \colon B_\h = A_\h \otimes A_{1/\h} \to C_u(\R)\rtimes \R_d,\]
The representation \(\rho\) is injective if \(\h\in \R\setminus \Q\).

\end{lemma}

The C*-algebra \(A_\h \otimes A_{1/\h} =:B_\h\) is the crossed product of \(C(\TT^2)\) by the group \(\Z^2\) with action 
\[ (n,m) \cdot (x,y) = (x+ n\h, y+\frac{m}{\h}).\]
and the homomorphism \(\rho\) embeds \(B_\h\) into \(C_u(\R)\rtimes \R_d\) by 
letting \(f\in C(\TT^2)\) map to \( f_p (t) := f( t, \frac{t}{\h})\), and embedding 
\(\Z^2\) isomorphically to the dense subgroup 
\[ \Lambda := \{ n \h + m \; | \; n,m \in \Z\} \subset \R.\] 
We can consider the \(\Z^2\) action as factoring through the translation action of the 
subgroup \(\Lambda\) acting through the Kronecker flow 
\[\alpha_t (x,y) = (x + t, y+\frac{t}{\h})\] along lines of slope 
\(1/\h\), because 
\[ \alpha_{ n\h +m} (x,y) = (x +  n\h ,  y+ \frac{m}{\h}) .\]

The restriction of \(\rho\) to \(C(\T^2)\) is thus a special case of Lemma \ref{pullbacktoflow} with

\begin{proposition}
\label{sldkjfos8rufsidjfls}
Let \(\gamma \colon \R \to \T^2\) be the group homomorphism 
\(\gamma (t) = (\h t, t)\), let \(U\) be its image. 
\begin{itemize}
\item[a)] \(U\) is dense if \(\h \notin \Q\). 
\item[b)] An element \( (x,y) \in \R^2\) projects to an element of \(U\) if and only if 
\(\h y = x + n +  m\h\) for some integers \(n,m\) if and only if \( \h y = x \) mod \(\Lambda \subset \R\). 
\item[c)] The subgroup \(\Lambda\) is contained in \(U\) for all integers \(n,m\). 
\item[d)] If \(U' := \gamma' (\R)\) with \(\gamma' (t) = ( t, \h t)\) then \(\Lambda = U\cap U'\).

\end{itemize}

\end{proposition}

\begin{proof}
a)-c) are routine. 
The (dense) subgroup \(\Lambda\subset \TT^2\) is obtained as follows. The 
line of slope \(1/\h\) in \(\R^2\) through the origin intersects the vertical 
lines \(x = n\), for \(n \in \Z\), in the points \(\Z^2\)-congruent to 
\( (n \h, 0)\), and through the horizontal lines \(y = m\), in the points 
congruent to \( (0, \frac{m}{\h})\).  Summing all of these points in \(\T^2\) gives 
\(\Lambda\), which is contained in \(U \) by c). The flip \(\R^2\to \R^2\) interchanges 
lines of slope \(\h\) and of \(1/\h\) and leaves \(\Z^2\) invariant 
interchanging vertical and horizontal lines. The last statement follows from symmetry.

\end{proof}

The subgroup \(\Lambda\) consists therefore of all points of \(\T^2\) in the intersection of the 
two dense subgroups \(U\) and \(U'\), projections of lines of slope \(\h\) and \(1/\h\). 
We call \(\Lambda\) the \emph{homoclinic subgroup}.

\begin{definition}
\label{bicycle}
The \emph{Heisenberg bi-cycle} is the spectral triple over
 \(C^\infty(\T^2)[ \Lambda] \subset A_{\h}\otimes
A_{1/\h}\) obtained by pulling back the Heisenberg cycle of Proposition \ref{universalspectrual} by the 
*-homomorphism \(\mu \colon C(\T^2)\rtimes \Lambda \to C_u(\R)\rtimes \R_d\) of Lemma \ref{whatismu}
(restricted to \(C^\infty(\T^2)[\Lambda]\)). 

The class of the Heisenberg cycle is denoted \(\Delta_\h \in \KK_0(A_\h\otimes A_{1/\h}, \C)\) 
(see \eqref{eq:heis} with \(\pi\) replaced by \(\rho:= \pi \circ \mu\)). 
\end{definition}

 In the next section we will show that the zeta functions \(\trace (\rho (a)H^{-s})\) extend meromorphically 
to \(\C\) for \(a\in C^\infty(\T^2)[\Lambda]\subset C(\T^2)\rtimes \R_d\), provided that 
\(\h\) satisfies a Diophantine condition. 

The
 inclusion \(A_\h \to A_\h \otimes A_{1/\h}\) pulls the Heisenberg bi-cycle back to 
a spectral cycle for \(\KK_0(A_\h, \C)\), which is a spectral triple over the 
smooth subalgebra \(A_\h^\infty\). As it is of special interest to us, we single it out in a definition.

\begin{definition}
\label{definition:ds}
The \emph{Heisenberg cycle}  is the even, \(2\)-dimensional
spectral cycle  
\[ \left(L^2(\R) \oplus L^2(\R), \pi_\h \oplus \pi_\h,\;\; D = \begin{bmatrix} 0 & A^* \\ A & 0\end{bmatrix}\right),\]
for \(\KK_0(A_\h, \C)\), defining a spectral triple 
over the Schwartz subalgebra \(A_\h^\infty\) of the 
rotation algebra \(A_\h := C(\TT)\rtimes_\h \Z\). 

The class in  \(\KK_0(C(\TT)\rtimes_\h \Z, \C)\) of the Heisenberg cycle is denoted \([D_\h]\). 
\end{definition}

\begin{remark}
The injection \(A_{1/\h}\) into \(A_\h \otimes A_{1/\h}\) pulls the Heisenberg bi-cycle back to 
a cycle and class \([D^\h]\) for \(\KK_0(A_{1/\h}, \C)\). But the unitary 
\(U\colon L^2(\R) \to L^2(\R)\), \(U\xi (x) = \sqrt{\h}\,  \xi (\h x) \) conjugates the 
representation \(\pi_{1/\h}\) to the representation \(\pi^{\h}\) (in notation of Lemma \ref{whatismu}). 
This effects the operator by a homotopically trivial re-scaling, and hence 
\([D^\h] = [D_{1/\h}] \in \KK_0(A_{1/\h}, \C)\).  
\end{remark}

If \(0<h<1\) then the spectral triple describing \([D_\h]\)
 is studied 
in \cite{LM}. It is related in an exact way to Connes' Dolbeault class (for any \(\h\)) as we now establish, 
although the result is already proved in \cite{LM} (for \(0<\h<1\), and
 in slightly different language).
 Let \(L^2(A_\h)\) denote 
the GNS Hilbert space associated to the trace \(\tau \colon A_\h \to \C\). On 
\(L^2(A_\h)\) the derivations \(\delta_1, \delta_2\) are defined \(\delta_1 (U) = 2\pi i U\), 
\(\delta_1 (V) = 0\), \(\delta_2 (U) = 0\), \(\delta_2 (V) = 2\pi i V\). Then 
the derivations assemble to give 
\[ \dol := \begin{bmatrix} 0 & \delta_1 - i\delta_2\\ \delta_1 + i\delta_2 & 0 \end{bmatrix}.\] 
The GNS representation \(\lambda \colon A_\h \to \Bound\left( L^2(A_\h)\right)\) then 
fits into a spectral triple over \(A_\h^\infty\) with Hilbert space \(L^2(A_\h)\oplus L^2(A_\h)\) and 
representation \(\lambda \oplus \lambda\). 

We let \([\dol]\in \KK_0(A)_\h, \C)\) be its class. Since \(\tau\) is not just a state, but a trace, 
the right multiplication operation of \(A_\h\) on itself determines another, commuting 
representation 
\(\lambda^{\mathrm{op}} \colon A_\h^{\mathrm{op}} \to \Bound\left (L^2(A_\h)\right)\). As \(A_\h\) is naturally 
isomorphic to its opposite algebra we obtain a pair of commuting representations of 
\(A_\h\) on \(L^2(A_\h)\) and on \(L^2(A_\h)\oplus L^2(A_\h)\). 
These observations determines a cycle and class 

\[ \Delta_{\dol} \in \KK_0(A_\h \otimes A_\h, \C).\]

Note that all of this only depends on the class of \(\h\) mod \(\Z\).

Connes proves that cup-cap product 
\begin{equation}
\label{equation:conneduality}
\PD_{\dol} \colon \KK_*(D_1, A_\h \otimes D_2) \to \KK_*(A_\h\otimes D_1, D_2) 
\end{equation}
(for any \(D_i\)) with 
\(\Delta_{\dol}\) induces an isomorphism, that is, yields a
 self KK-duality for \(A_\h\). 
The result is refined in 
\cite{DE}.

\begin{lemma}
 \label{theorem:reiffelsmeb}
 Define on \(C_c(\R)\) the inner products 
 \[ _{C(\R/\h \Z) \rtimes \Z} \langle \xi, \eta\rangle ( x, m) = 
 \sum_{n \in \Z} \xi ( x-n\h) \cdot \overline{\eta (x-n\h -m)}, \;\; x\in \R/\h \Z, \;\; m \in \Z,\]
 and 
 \[ \langle \xi, \eta \rangle_{C(\R/\Z)\rtimes_\h \Z} (x, m) = 
 \sum_{n \in \Z}  \overline{\xi ( x-n)} \eta ( x-n-m\h), x\in \R/\Z, \;\; m \in \Z\]  
 Give \(C_c(\R)\) the \(C(\R/\h \Z) \rtimes \Z\)- \( C(\R/\Z)\rtimes_\h \Z\) bimodule structure 
 with 
 \[ ( n\xi) (x) = \xi ( x-n), \;\;\ ( f\xi ) (x) = f(x) \xi (x), \;\;\;\;\;\;\; (\xi n ) (x) = \xi (x+n\h), \;\; ( \xi f) (x) = f(x) \xi (x).\]
 Then \(C_c(\R)\) completes to a Morita equivalence 
  \(C(\R/\h \Z) \rtimes \Z\)- \( C(\R/\Z)\rtimes_\h \Z\) bimodule \(\Hilb_\h\),  that is, 
  to a Morita equivalence \(A_{1/\h}\)-\(A_\h\)-bimodule. 
  
   \end{lemma}

The relation between Connes' Dolbeault class \([\dol]\) and the Heisenberg \([D_\h]\) is 
based on the following simple relationship between modules.

\begin{lemma}
\label{t4ensorproductcomputation}
The tensor product of Hilbert modules 
\(\Hilb_\h\otimes_{A_\h} L^2(A_\h)\) over the representation \(\lambda \colon A_\h \to \Bound (L^2(A_\h))\), is 
naturally isomorphic to \(L^2(\R)\) as a Hilbert space. 

Under this identification: 

\begin{itemize}
\item[a)] The representation \(\lambda\) of \(A_{1/\h}\) 
on \(\Hilb_\h \otimes_{A_\h} L^2(A_\h)\) induced by its representation on \(\Hilb_\h\)
corresponds to the representation \(\pi^\h\) on \(L^2(\R)\) of Lemma \ref{whatismu}. 
\item[b)] The representation \(\lambda^{\mathrm{op}}\) of 
 \(A_\h\) on \(L^2(A_\h)\) commutes with the representation \(\lambda\) involved in the tensor product. 
 Hence \(A_\h\) is also represented on \(\Hilb_\h\otimes_{A_\h} L^2(A_\h)\) by \(1\otimes \lambda^{\mathrm{op}}\). This representation 
 identifies with \(\pi_\h\) on \(L^2(\R)\) of Lemma \ref{whatismu}. 
\end{itemize}

\end{lemma}

\begin{proof}
If \(f_1, f_2 \in C_c(\R)\), then their \(A_\h = C(\R/\Z)\rtimes_\h \Z\)-valued inner product is 
given in the above Lemma. Let \(\delta_0 \in L^2(A_\h)\) the vector corresponding to \(1\in A_\h\)
 and
 consider the elements \(f_i\otimes \delta_0 \in \Hilb_\h\otimes_{A_\h} L^2(A_\h)\). 
Their inner product is given by 
\begin{multline}
\langle f_1\otimes \delta_0 , f_2\otimes \delta_0\rangle 
=  \langle \delta_0, \langle f_1, f_2\rangle_{A_\h} \,\delta_0\rangle
 = \tau \left( \langle f_1, f_2\rangle\right) = \int_0^1 \langle f_1, f_2\rangle_{A_\h} (x, 0) dx\\ = \langle f_1, f_2\rangle_{L^2(\R)},
\end{multline}
where \(\tau\colon A_\h \to \C\) is the trace. It follows that \(f\mapsto f\otimes \delta_0\) induces a 
Hilbert space isometry \(L^2(\R) \to \Hilb_\h \otimes_{A_\h} L^2(A_\h)\). Since elements of the 
form \( \xi \otimes \delta_0\), \(\xi \in\Hilb_\h\), are dense in the tensor product (because the GNS representation is 
cyclic), this isometry is actually a unitary. The other statements are easy to check. 

\end{proof}

\begin{corollary}
\label{technicalstuff}
Let \([\Hilb_\h ]\in \KK_0(A_{1/\h}, A_\h)\) be the class of the Morita equivalence bimodule 
\(\Hilb_\h\),  \(\Delta_\h\) the class of the Heisenberg bi-cycle (Definition \ref{bicycle}) and 
 \(\PD_{\dol}\) be Connes' Poincar\'e duality\eqref{equation:conneduality}.
 Then 
 \begin{itemize}
 \item[a)] \( \PD_{\dol} ([\Hilb_\h]) = [\Delta_\h]\in \KK_0(A_\h\otimes A_{1/\h}, \C),\)
 \item[b)] The class \(\Delta_\h\in \KK_0(A_\h\otimes A_{1/\h}, \C)\) determines a KK-duality between \(A_\h\) and \(A_{1/\h}\). 
 \item[c)] If \([p_\h]\in \K_0(A_\h)\) denotes the class of the Rieffel projection then 
 \(\PD_{\dol} ([p_\h]) = [D_h] \). 
 \end{itemize}

\end{corollary}

\begin{proof}
\(\PD_{\dol} ([\Hilb_\h]) = (1_{A_\h}\otimes [\Hilb_\h]) \otimes_{A_\h\otimes A_{\h}} \Delta_{\dol}\in 
\KK_0(A_\h \otimes A_{1/\h}, \C)\). 
by definition. The module composition involved in the Kasparov product 
results in (two copies of) 
\(L^2(\R)\) with (two copies of) the Heisenberg representation
\(\rho_\h\) of Theorem \ref{whatismu}, by Lemma \ref{t4ensorproductcomputation}. 
The operator \(D\) satisfies the connection condition for the axiomatic approach to the 
product by \cite{LM}.

Let \(\PD_\h\) denote the analogue of \eqref{equation:conneduality} using 
\(\Delta_\h\) in place of \(\Delta_{\dol}\). Then for \(x\in\K_*(A_{1/\h})\), \(y \in \K_*(A_{\h})\), 
\[ \langle \PD_\h (x) , y\rangle = \langle y\otimes_\C x, \Delta_\h\rangle = 
\langle y\otimes_\C x ,  (1_{A_\h}\otimes [\Hilb_\h]) \otimes_{A_\h \otimes A_\h} \Delta_{\dol}\rangle
= \langle y \otimes_\C \Hilb_\h^*(x) , \Delta_{\dol}\rangle .    \]

Since \([\Hilb_\h]\) is an equivalence in \(\KK\), the intersection form for \(\Delta_\h\) is obtained 
by twisting the form for \(\Delta_{\dol}\) by an isomorphism, and hence is non-degenerate, since 
Connes' is. 

By definition 
\begin{multline}
\PD_{\dol}([p_\h]) 
= \left( [p_\h]\otimes 1_{A_\h}\right) \otimes_{A_\h\otimes A_\h} \Delta_{\dol}
= (u\otimes 1_{A_\h})^*\left( \Hilb_\h \otimes 1_{A_\h}\right) \otimes_{A_\h\otimes A_\h} \Delta_{\dol}
\\= (u\otimes 1_{A_\h})^*( \Delta_\h) = [D_\h]
\end{multline}
where \(u \colon \C \to A_{1/\h}\) is the unital inclusion, where the non-trivial step was the 
penultimate one, which used c). 
\end{proof}

We are going to show using 
cyclic cohomology calculations that
\( \langle [p_\h], [D_\h]\rangle 
 = -\left \lfloor{\h}\right \rfloor \).

 This is enough to describe the intersection form induced by \(\Delta_\h\). 
We note the result for the record here. 

\begin{proposition}
\label{indexpaiinrgs}
For any \(\h\) give \(\K_0(A_\h)\) the ordered free abelian group basis \(\{[1], [p_\h]\}\).
Then the matrix of the intersection form induced by \(\Delta_{\h}\) is 
\[\begin{bmatrix} 1 & -\left \lfloor{1/\h}\right \rfloor  \\ -\left \lfloor{\h}\right \rfloor  
& 1\end{bmatrix} \]

\end{proposition}

\begin{proof}
By the definitions \( \langle \PD_\h ([p_\h]), [p_{1/\h}]\rangle = 
\langle [p_\h]\otimes_\C [p_{1/\h}] , \Delta_\h\rangle \). As noted above 
\(\Delta_\h = \PD_{\dol} ([\Hilb_\h]) 
:= (1_{A_\h}\otimes [\Hilb_\h]) \otimes_{A_\h\otimes A_{\h}} \Delta_{\dol}\) so this may be 
written
\[ \langle [p_\h]\otimes_\C [p_{1/\h}] , (1_{A_\h}\otimes_\C [\Hilb_\h])^*(\Delta_{\dol})\rangle\]
Moving \([\Hilb_\h]\) to the other side and noting that 
\( [p_{1/\h}]\otimes_{A_{1/\h}} [\Hilb_\h] = 
[1] \in \K_0(A_{\h})\) gives that 
\begin{multline}
  \langle \PD_\h ([p_\h]), [p_{1/\h}]\rangle = \langle [p_\h]\otimes_\C [p_\h], \Delta_{\dol}\rangle
= \langle \PD_{\dol} ([p_\h]), [1]\rangle = \langle [D_\h], [1]\rangle
= 1.
\end{multline}

\end{proof}

We end by the remark that Heisenberg cycles, over \(C_u(\R)\rtimes \R_d\), are 
only \emph{topologically} interesting if \(G\) is non-trivial. L\"uck and Rosenberg 
construct a homotopy in KK-theory by considering the operators 
\(\lambda x + d/dx\) for \( \lambda \in [1,\infty)\). This field can be continuously 
extended to \([1,\infty]\) by adding a copy of \(\C\) to \(L^2(\R)\) at infinity, and 
extending the operator by the direct sum of the 
multiplication operator \(x/\abs{x}\) on \(L^2(\R)\), and \(0\) on the \(1\)-dimensional 
summand.  Their argument implies the following.

\begin{proposition}
The class in \(\KK_0(C_u(\R), \C)\) of the Heisenberg cycle over \(C_u(\R)\), 
 is equal to the 
class \([\ev_0]\in \KK_0(C_u(\R), \C)\), 
of the point-evaluation homomorphism
\(C_u(\R) \to \C\), \(f\mapsto f(0)\).

In particular, if \([D_\alpha]\) is the Heisenberg cycle for an ergodic flow on \(M\), 
then \([D_\alpha]) = [\ev_p]\in \KK_0(C(M), \C)\).

\end{proposition}

This shows that it is essential to consider the
 crossed products \(C_u(\R)\rtimes \Gamma\), for 
suitable \emph{non}-trivial groups \(G \subset \R\), in order 
to see interesting topological phenomena. 

However, the \emph{geometry} of the Heisenberg cycles is 
by contrast interesting, even without taking into account a group action, as
we discuss in the next section.

\section{Zeta functions and ergodic flows}

 Let \(f\in C_u(\R)\) be a bounded, uniformly continuous function. 
 
 We consider 
the  zeta function \(\trace (fH^{-s})\) where \(H \) is the harmonic oscillator. 
\eqref{harmonicoscillator}, which is analytic for \(\mathrm{Re} (s)>1\). 
 
\begin{theorem}
\label{maintheorem}
If \(f\in C_u(\R)\) then the difference of analytic functions on \(\mathrm{Re} (s)>1\)
\begin{equation}
\label{equation:integralgormul1a}
 \Gamma (s) \cdot \Tr (f H^{-s}) 
- 
\frac{1}{2\sqrt{\pi}} \cdot  \int_0^1 \int_\R  
 t^{s-1} \cosech t  \cdot f(x \sqrt{\coth t})  \cdot  
e^{-x^2} dxdt 
\end{equation}
extends analytically to \(\C\). 
\end{theorem}
 
 \begin{remark}
 \label{constantterm}
If \(f = 1 \) is constant  \eqref{equation:integralgormul1a} gives that 
\[ \Gamma (s) \cdot \trace (H^{-s}) = 
 \int_0^1 \int_\R  
 t^{s-1} \cosech t 
dt \]
up to an entire function. 
The Mellin transform of \( \cosech t \) is \( 2(1-2^{-s})\Gamma(s)\zeta(s)$, 
where \(\zeta\) is the Riemann zeta function. 
This is meromorphic on the whole complex plane; dividing by $\Gamma(s)$, we get a 
single simple pole at $s=1$. The residue there is 
equal to \(1\). 
\end{remark}

 Before proving the Theorem we discuss applications.

If \(f\in C_u(\R)\), let 
 \(F(T) = \int_0^Tf(t)dt\).
 
  \(F\) is uniformly continuous, not necessarily bounded, 
 but \(\abs{F(T)} = \mathrm{O}(T)\) as \(T\to \infty\).

 \begin{lemma}
 \label{antiderivatives}
 If \(f\in C_u(\R)\) admits \(n\) successive bounded 
 anti-derivatives, \( ^1f, ^2f, \ldots , ^nf\), then 
 \(\trace (fH^{-s})\) extends analytically to \(\mathrm{Re} (s)>1-\frac{n}{2}\). 
 
 \end{lemma}

 \begin{proof}
By Theorem \ref{maintheorem} 
 \begin{equation}
\label{equation:integralgormula2}
 \Gamma (s) \cdot \Tr (f H^{-s}) 
\sim  
\frac{1}{2\sqrt{\pi}} \cdot  \int_0^1 \int_\R  
 t^{s-1} \cosech t   \cdot \int_\R  f(x \sqrt{\coth t})  \cdot  
e^{-x^2} dxdt 
\end{equation}
where \(\sim\) means up to an entire function. 
Let \(F = ^1f\), then integration by parts gives 
 \begin{multline}
= \frac{1}{\sqrt{\pi}} \cdot  \int_0^1 \int_\R  
 t^{s-1} \cosech t \sqrt{\tanh t}  \cdot F(x \sqrt{\coth t})  \, 
xe^{-x^2} dxdt \\ = \frac{1}{\sqrt{\pi}}\int_0^1 t^{s-1} \cosech t \sqrt{\tanh t}\cdot  \phi (t) dt,
\end{multline}
 with \(\phi (t)  = \int_\R F(x \sqrt{\coth t})  \, 
xe^{-x^2} dx\).
 The function \(t^{s-1} \cosech t \sqrt{\tanh t} \cdot \phi (t)\) is \( \sim
  t^{s- 3/2}\cdot \phi (t)  \) as \(t\to 0\), and is integrable over \([0,1]\) 
 for \(\mathrm{Re} (s) >\frac{1}{2}\) if \(\phi \) is continuous and bounded as \(t\to 0\). 
  In particular this holds if \(F\) is bounded on \(\R\). So we have verified analyticity for \(\mathrm{Re} (s)>\frac{1}{2}\). Repeating the argument, if \(^2f = ^1F\) is the second anti-derivative then the previous expression can be written 

\begin{equation}
 \frac{1}{\sqrt{\pi}} \cdot  \int_0^1  
 t^{s-1} \cosech t \tanh t  \cdot \int_\R \; ^2f (x \sqrt{\coth t})  \, 
 (1-2x) e^{-x^2} dxdt 
\end{equation}
which is analytic now for \(\mathrm{Re} (s) >0\) if \(^2f\) is also bounded. 
One repeats this argument \(n\) times and the statement follows.

 \end{proof}

The \emph{cohomological equation} in dynamics refers to the differential equation 
\[ Xu = f\]
where \(X\) is a generating vector field for a smooth flow \(\alpha\) 
on a compact manifold 
\(M\). Let \(p\in M\), \(f\in C(M)\) and 
\[ f_p (t) = f\left(\alpha_t (p)\right).\]
Then \(f_p \in C_u(\R)\). If \(f\in C^\infty (M)\) then \(f_p\in C_u^\infty (\R)\).

An obstruction to solving the cohomological equation for given \(f\) 
is the mean of \(f\) with respect to any \(\alpha\)-invariant probability measure \(\mu\). 
This follows from 
differentiating the equation 
\(  \int_M u \circ \alpha_t \; d\mu = \int_M u d\mu\), 
which gives that \(\int_M Xu \; d\mu = 0\), that is, \(\int_M fd\mu = 0\) if 
\(Xu = f\) has a solution.

Conversely, if one can solve \(Xu = f\) for given \(f\), then \(u_p (t) := u\left(\alpha_t (p)\right)\) 
supplies a bounded anti-derivative of \(f_p\). Hence Lemma 
\ref{antiderivatives} gives the following.

\begin{proposition}
\label{proposition:dyanmicsrfjsef}
Let \(\alpha\) be a smooth flow on \(M\) with generator \(X\) 
and \(\mu\) any \(\alpha\)-invariant measure. If \(f\in C^\infty(M)\) and 
\(\int_M fd\mu = 0\) implies 
that \(Xu = f\) for some \(u \in C^\infty (M)\), then \(\trace (f_pH^{-s})\) extends 
meromorphically to \(\C\) and 
\[ \Res_{s=1}\trace (f_pH^{-s}) = \int_M fd\mu\]
for any \(p\in M\).

\end{proposition}

In some simple situations of \emph{elliptic dynamics}, \emph{e.g.} the 
periodic flow on the circle, having mean zero is the only obstruction to solving the 
cohomological equation for \(f\):  indeed, in this case one can make a 
extremely strong statement not even requiring smoothness. 

\begin{lemma}
Let \(f\) be continuous and \(\rho\)-periodic on \(\R\) with zero mean: \(\int_0^\rho f(t) dt = 0\). Then 
 \(F(T)\) is also continuous,  \(\rho\)-periodic, with zero mean. 

\end{lemma}

 \begin{corollary}
 If \(f\) is continuous and \(\rho\)-periodic then \(\trace (fH^{-s})\) meromorphically extends to 
 \(\C\) with a simple pole at \(s = 1\) and 
 \[ \Res_{s=1} \trace (fH^{-s}) = \frac{1}{\rho} \int_0^\rho f(t) dt   =:\mu (f) .\]
 
 \end{corollary}
 
 \begin{proof}
 \(\bar{f} := f- \mu (f)\) has zero mean. Applying the previous lemma gives that \(f\) has 
 bounded anti-derivatives of all orders; the result follows from Lemma \ref{antiderivatives}.
and Remark  \ref{constantterm}
 \end{proof}

 \begin{definition}
 \label{defin:dioflow}
 
 Let \(\alpha\) be a smooth, ergodic  Riemannian flow on a compact Riemannian manifold 
 \(M\) (\(\alpha_t \colon M \to M\) is a Riemannian isometry for all \(t\).) Let 
 \(\Delta\) be the Laplacian on \(M\), \(0 = \lambda_0 < \lambda_1 < \cdots\), 
 \(\mu\) normalized volume measure on \(M\) and
  \(L^2(M) = \oplus_{n=0}^\infty H_n\) 
 the \(\Delta\)-spectral decomposition of \(L^2(M)\), \(H_n = \ker(\lambda_n - \Delta)\). 

 The vector field \(X\) commutes with \(\Delta\) as an operator on \(C^\infty (M)\) and 
 so leaves each \(H_n\) invariant. For \(n>0\), 
 \(\epsilon_n = \norm{X|_{H_n}}\). Since \(\alpha\) is 
 ergodic, the kernel of \(X\) consists of constant functions, and, moreover, 
 \(H_0 \) is the space of constant functions. 
 
 A Riemannian flow \(\alpha\) satisfies a \emph{Diophantine condition} if there exists \(C \ge 0\) and 
 \(\gamma >0\) such that \( \epsilon_n \ge Cn^{-\gamma}\). 
 \end{definition}
 
 \begin{corollary}
 If \(\alpha\) is a smooth, Riemannian, ergodic flow on \(M\) satisfying a Diophantine condition, 
 and \(f\in C^\infty (M)\) then \(\trace (f_pH^{-s})\) 
 extends meromorphically to \(\C\) with a simple pole 
 at \(s=1\) and \[\Res_{s=1} \trace (f_pH^{-s}) = \int_Mf d\mu\]
  for any \(\alpha\)-invariant measure \(\mu\) and any \(p\in M\).

 \end{corollary}

 \begin{proof}
 Proceeding as in the discussion above, let \(f\) be smooth on \(M\), then 
 \(f\in L^2(M)\) and \(f = \sum_{n=0}^\infty f_n \) where \(s_n\) are \(\lambda_n\)-
 eigenvectors for \(\Delta\).  The linear operators \(e_n = X|_{H_n}\) have no kernel for 
 \(n>0\) because the flow is ergodic, and \(e_0 = 0\). Note that 
 \(f_0 = \int_M fd\mu\). Assuming that this is zero, we can set 
 \[ u := \sum_{n=1}^\infty  e_n^{-1} f_n \, s_n.\]
 If \(f\) is smooth, the sequence \( \{\norm{f_n}\}\) has rapid decay. The 
 Diophantine assumption implies that \(\{e_n^{-1} f_n\}\) also has rapid decay, and 
 hence defines a smooth function on \(M\). 
 
 This shows that the only obstruction to solving the cohomological equation
 \(Xu = f\) for \(f\) smooth, is \(\int_M f d\mu = 0\). The result follows from 
 Lemma \ref{proposition:dyanmicsrfjsef}.

 \end{proof}

 The hypothesis holds if \(\h \in \R\setminus \Q\) is an irrational number satisfying 
 a Diophantine condition, 
  and \(f\in C^\infty (\T^2)\), \(f_p (t) = f(\alpha_tp)\) with \(\alpha_t (x,y) = (x + t, y+\h t)\) 
  Kr\"onecker flow. Then \(X = \frac{\partial}{\partial x} + \h \frac{\partial }{\partial y} \) acts 
  on the eigenfunctions \(z^nz^m\) for \(\Delta\) on \(\T^2\) by the constant 
  \(n+\h m\). The usual Diophantine condition on an irrational number gives 
  \(\gamma\) such that \( \abs{n+\h m}\ge C (n^2+m^2)^{-\frac{\gamma}{2}}\), and 
  this implies the flow is Diophantine in the broader sense above.

\begin{corollary}
\label{corollarydiomero}
Let \(\omega = (\omega_1, \ldots , \omega_n) \in \R^n\) where \(\omega_1, \ldots , \omega_n)\) 
are rationally independent. Assume the following Diophantine condition. 
There exists \(C >0\) and \( \gamma >0\) such that 
\( \abs{\sum_{i=1}^n k_i \omega_i } \ge C\abs{k}^{-\gamma}\), with \(\abs{k}\) the word 
length of \(k = (k_1, \ldots, k_n)\) in \(\Z^n\). Then if \(\alpha_t (x) = x+t\omega\) is the 
corresponding linear flow on \(\TT^n\), and \(f\in C^\infty (\T^n)\), then for any \(p\in \T^n\), 
\(\trace (f_pH^{_1})\) extends meromorphically to \(\C\) with a simple pole at \(s=1\) and 
\[\Res_{s=1}\trace (fH^{-s}) = \int_{\T^n} fd\mu,\]
\(\mu\) Lebesgue measure on \(\T^n\).

\end{corollary}

We now proceed to the proof of Theorem \ref{maintheorem}. A computation of the
 heat kernel of \(e^{-tH}\) follows from solving a differential equation: the heat kenel \(k_t\) 
 satisfies 
\( ( \frac{\partial}{\partial t} + H)\cdot \phi_t = 0,\)
where 
\(\phi_t (x) = \int_\R k_t(x,y) \phi (y)dy,\)
for \(\phi \in \Schwartz (\R)\), and \(t \ge 0\), together with the initial 
condition 
\(\lim_{ t\to 0} \phi_t = \phi.\)
Consider the \emph{ansatz}
\[ k_t (x,y) = \exp \left( \frac{a_t}{2}x^2+b_txy+\frac{a_t}{2}y^2+ c_t\right).\]
Setting this equal to \(0 \) and solving for coefficients
 gives the ordinary differential equations 
\[ \frac{\dot{a}_t }{2} = a_t^2-1 =b_t^2,\;\;\; \dot{c}_t^2 = a_t.\]
Solving these gives 
\[ a_t = - \coth (2t+C) , \;\;\; b_t = \cosech(2t+C), \;\;\; c_t = -\frac{1}{2} \log \sinh (2t+C)+D.\]
Using the initial conditions we get \(C = 0\) and \(D = \log (2\pi )^{-\frac{1}{2}}\). 
See \cite{Berline}.

We obtain the following, called \emph{Mehler's formula} \cite{Mehler}.

 \begin{lemma}
 \label{lemmeyler}
 \begin{equation}
\label{equation:hetkernelsintehsfdwlow53}
k_t (x,y) = \\ \frac{1}{\sqrt{2\pi \sinh 2t }}
\exp\left( -\tanh t \cdot \frac{ (x+y)^2}{4} - \coth t \cdot \frac{(x-y)^2}{4}\right) 
\end{equation}
\end{lemma}
 
 \begin{proof} (Of Theorem \ref{maintheorem}.  The operator \(H^{-s}\) is trace-class for 
 \(\mathrm{Re} (s)>1\), and the operator-valued integral \(\int_0^\infty t^{s-1} e^{-tH} dt\) 
 converges in norm to \(\Gamma (s) \cdot H^{-s}\). Hence if \(a\in \Bound (L^2\R)\), 
 \begin{equation}
 \label{integralsexpresjdnskfjsd}
 \Gamma (s) \cdot a H^{-s} = \int_0^\infty t^{s-1} ae^{-tH}dt.\end{equation}
 and taking traces gives 
  \begin{equation}
 \label{integralsexpresjdnskfjsd2}
 \Gamma (s) \cdot \trace ( a H^{-s})  
 = \int_0^\infty t^{s-1} \trace ( ae^{-tH}) \; dt.\end{equation}
 
 Furthermore, if \(a\) is any bounded operator then  
 \[ \int_1^\infty t^{s-1} \trace (ae^{-tH}) dt\]
 extends to an analytic function on \(\C\). Hence
  \begin{equation}
 \label{integralsexpresjdnskfjsd5}
 \Gamma (s) \cdot \trace ( a H^{-s})  
 -  \int_0^1 t^{s-1} \trace ( ae^{-tH}) \; dt.\end{equation}
 extends analytically to \(C\).

 Now let \(f\in C_u(\R)\), set \(a = f\). Then \(fe^{-tH}\) is an integral operator with 
 kernel \(f(x) k_t(x,y)\), and hence \(\trace (fe^{-tH}) = \int_\R f(x)k_t(x,x)\; dx.\) 
 Applying Mehler's formula Lemma \ref{lemmeyler} gives 
 
 \begin{equation}
 \label{integralsexpresjdnskfjsd2}
 \Gamma (s) \cdot \trace ( a H^{-s})  
\\ = \int_0^\infty t^{s-1}  \frac{1}{\sqrt{2\pi \sinh 2t }}\int_\R
f(x) e^{ -x^2 \tanh t}   \; dxdt.
 \end{equation}
Making the
 change of variables \( x \mapsto \frac{x}{\sqrt{\tanh (t)}} \) gives 
 \begin{equation}
 \label{integralsexpresjdnskfjsd3}
 \Gamma (s) \cdot \trace ( a H^{-s})  
\\ = \int_0^\infty t^{s-1}  \frac{\sqrt{\coth t}}{\sqrt{2\pi \sinh 2t }}\int_\R
f(x\sqrt{\coth t}) e^{ -x^2 }   \; dxdt.
 \end{equation}
The result follows from the identity 
\(\frac{\coth t}{\sinh 2t} = \cosech^2 t\).

 \end{proof}
 
 If \(\trace (fH^{-s})\) meromorphically extends past \(\mathrm{Re} (s) = 1\), then 
 the residue of the pole at \(s=1\) defines kind of `mean' of \(f\in C_u(\R)\). In certain 
 examples of flows where \(f = g_p\) for \(p\in M\), \(f(t):= g(\alpha_t p)\), 
 we have noted (Proposition \ref{proposition:dyanmicsrfjsef})
  that this spectrally defined mean agrees with 
 the geometric mean \(\int_Mf d\mu\) over the manifold. 
 
 The spectrally defined mean does not require meromorphic 
 continuation, but only existence of the limit \(\lim_{s\to 1^+} (s-1) \, \trace (fH^{-s})\), 
 which is a weaker condition which we now discuss.

 \begin{definition}
 Let \(\D\subset C_u(\R)\) be the closed linear subspace
 of all \(f\) such that 
 \begin{equation}
 \label{eqiuation:restrace}
 \restrace (f) := \lim_{s\to 1^+} (s-1) \, \trace (fH^{-s}) 
 \end{equation}
 exists. 
 \end{definition}

The residue trace \(\restrace \) defines a positive linear functional 
of norm \(1\) on \(\D\). This is immediate from 
the following geometric description of \(\restrace\) given below. 
 
 \begin{theorem}
 \label{yay}
 If \(f\in C_u(\R)\) then \(f\in \D\) if and only if 
 \begin{equation}
 \label{restracegeometric}
\mu_u(f):= \lim_{\lambda \to \infty} \frac{1}{2\sqrt{\pi}}\int_0^1\int_\R f(x\, t^{-\lambda}) e^{-x^2} dx dt
 \end{equation}
 exists, and if this holds, then \(\restrace (f) = \mu_u (f)\).

 \end{theorem}

 \begin{proof}
 Choose \(\epsilon>0\). 
Since \(\Gamma (1) = 1\) by Theorem \ref{maintheorem} we have for \(\mathrm{Re} (s)>1\), 
\begin{equation}
\label{123}
\lim_{s\to 1+}\trace (fH^{-s}) = \lim_{s\to 1+} \frac{s-1}{2\sqrt{\pi}} \cdot  \int_0^1 \int_\R  
 t^{s-1} \cosech t  \, f(x \sqrt{\coth t})  \, 
e^{-x^2} dxdt 
\end{equation}
 In the proof, we noted that the part of the integral corresponding to \(t\ge \delta\) extends 
 analytically to \(\C\). Hence it contributes zero to the limit, and we may choose 
 \(\delta>0\) small enough that \(\abs{ t\cosech t - 1} <\epsilon\) for \( 0 <t <\delta\), so that 
 \(\abs{\cosech t - \frac{1}{t}} <\frac{\epsilon}{t}\) for \(t<\delta\). Let \(\phi_f (t) = 
\int_\R f(x \sqrt{\coth t})  \cdot  
e^{-x^2} dx\), then 
\begin{equation}
\label{12345}
\abs{ \int_0^\delta t^{s-1} \, (\cosech t - \frac{1}{t}) \, \phi (t) dt} 
< \frac{\epsilon}{s-1} \cdot  \norm{f}     
\end{equation}
by a brief computation. 
Letting \(\epsilon \to 0\) we see that the limit on the right hand side of \eqref{123}, if it exists, 
equals the limit 
\begin{equation}
\label{1234544}
\lim_{s\to 1+} \frac{s-1}{2\sqrt{\pi}} \cdot  \int_0^1 \int_\R  
 t^{s-2}  \, f(x \sqrt{\coth t})  \cdot  
e^{-x^2} dxdt .
\end{equation}
Let \(\lambda = \frac{1}{s-1}\) and substitute \(t\to t^\lambda\) in the above expression, and, 
noting \(\delta^{\frac{1}{\lambda}}\to  1 \) as \(\lambda \to \infty\), we deduce that 
 \begin{equation}
\label{1234544}
\mu_u (f) = \lim_{s\to 1+} (s-1) \trace (fH^{-s}) = 
\lim_{\lambda \to \infty} \frac{1}{2\sqrt{\pi}} \cdot  \int_0^1 \int_\R  
 f(x \sqrt{\coth t^\lambda})  \cdot  
e^{-x^2} dxdt .
\end{equation}

Since Lipschitz functions are dense in \(C_u^\infty (\R)\) and \(\D\) is closed, we may 
assume \(f\) is Lipschitz, and it follows that 

 \begin{multline}
\label{1234544}
\abs{ \int_0^1 \int_\R  
 \left( f(x \sqrt{\coth t^\alpha})  - f(xt^{ -\frac{\alpha}{2}}\right) e^{-x^2} dxt}
\\  \le \textup{const.}\lim_{\alpha \to \infty}
  \int_0^1  \abs{ \sqrt{\coth t^\alpha} - t^{-\frac{\alpha}{2}}} \, dt 
 \end{multline}
 which \(\to 0\) as \(\lambda \to \infty\). 
 This proves the result. \end{proof}

 The theorem can be expressed this way: 
 
 \begin{theorem}
 Let \(\mu_0 = \frac{1}{2\sqrt{\pi}} e^{-x^2}dx\), the Gaussian probability measure on 
 \(\R\). For \(t\in \R^*_+\) let  \(\rho_t \colon \R \to \R, \) \(\rho_t(x) = tx\), and \(\mu_t:= (\rho_t)_*\mu_0\). 
 Then 
 \[\lim_{\lambda \to \infty} 
  \int_0^1 \mu _{t^\lambda} \, dt  =  \restrace \in \D'\]
 in the weak topology on \(\D'\). 
 
 \end{theorem}

 We deduce the following.

 \begin{corollary}
 \label{corollaryzebigone}
If \(f\in C_u(\R)\)  and \( \lim_{T\to\pm  \infty} \frac{1}{T} \int_0^Tf (t)\, dt \) exists, then 
\( f\in \D\) and 
\begin{equation}
\label{birkhoff}
\restrace (f) =  \lim_{T\to\pm  \infty} \frac{1}{T} \int_0^Tf (t)\, dt .
\end{equation}
 
 In particular, if \(\alpha\) is an ergodic flow on a compact smooth manifold \(M\), \(\mu\) an 
 \(\alpha\)-invariant probability measure, 
 \(f \in C(M)\),  \(f_p (t) := f(\alpha_tp)\), then \(f_p\in \D\) and 
 \[  \restrace (f_p) = \int_M f\, d\mu \]
 for a.e. \(p\in M\).

 \end{corollary}

 \begin{proof}
 
 Integration by parts, the change of variables \(u\to ut^\lambda\), and a 
 slight re-arrangement gives 
  \begin{multline}
\label{1234544555}
\int_0^1 \int_\R  f(xt^{-\lambda}))e^{-x^2} dxdt 
= 2\int_0^1   \int_\R \int_0^x f(ut^{-\lambda}) xe^{-x^2} dudxdt 
\\=2 \int_0^1\int_\R t^{\lambda} \int_0^{xt^{-\lambda}} f(u)xe^{-x^2}dudxdt
=  2\int_0^1\int_\R  \frac{1}{xt^{-\lambda}} \int_0^{xt^{-\lambda}} f(u)x^2e^{-x^2}dudxdt
 \end{multline}
Now letting \(\lambda \to \infty\) and using 
the hypothesis that \(L:= \lim_{T\to\pm  \infty} \frac{1}{T} \int_0^Tf (t)\, dt \) exists gives 
 \begin{equation}
\label{12345445sfsdf55}
\lim_{\lambda \to \infty} \, \int_0^1 \int_\R  f(xt^{-\lambda}))e^{-x^2} dxdt 
= L\sqrt{\pi} 
 \end{equation}
 By Theorem \ref{yay} 
 \begin{equation}
\label{12345445sfwsfsdf55}
\restrace (f) = 
 \lim_{\lambda \to \infty} \frac{1}{2\sqrt{\pi}}\int_0^1\int_\R f(x\, t^{-\lambda}) e^{-x^2} dx dt
 \end{equation}
giving \eqref{birkhoff}. 

The second statement follows from combining the first with the Birkhoff Ergodic Theorem. 
 
 \end{proof}

 \begin{remark}
 By a slightly more elaborate argument the assumption of Corollary 
 \ref{corollaryzebigone} can be weakened to: there exists \( 0 \le \beta <1\) such that 
 \[L:= \lim_{x\to \pm \infty} \frac{1}{\int_0^x u^{-\beta}}\; \int_0^x f(u)u^{-\beta} \, du\]
 exists. Then \(f\in\D\) and \(\restrace (f) = L\) remains true. 
 
 \end{remark}

 We next produce some estimates related to group translation 
 operators on \(L^2(\R)\).

 \begin{lemma}
Let \(U_\alpha\) be the unitary induced by translation on the real line 
by \(\alpha \not= 0 \). Then if 
\(f \in C_u(\R)\) and \(a = fU_\alpha\in C_u(\R)\rtimes \R_d\) then 
\begin{equation}
\label{dos8fguseoifjsdf}
\Gamma (s) \cdot \Tr (fU_\alpha H^{-s})  =  
\frac{1}{2\sqrt{2\pi}} \int_0^\infty t^{s-1} \cosech t \exp (-\frac{\alpha^2}{4} \coth t) \, \mu_{\alpha, t} (f) dt
\end{equation}
where 
\(\mu_{\alpha, t} (f) = \int_\R f(x \sqrt{\coth t}  + \alpha) e^{-x^2} dx.\) 
\end{lemma}

\begin{proof}
The argument proceeds as in the proof of Lemma \ref{lemmeyler}. 
The operator \(fU_\alpha e^{-tH}\) is a compact integral operator with kernel 
\[ k_t'(x,y) = f(x) k_t(x-\alpha, y).\]

Hence for \(\mathrm{Re} (s)>1\), 
\begin{multline}
\Gamma (s) \cdot \Tr (fU_\alpha H^{-s}) = 
 \int_0^\infty \int_\R t^{s-1} k_t(x-\alpha,x) dxdt \\= 
 \int_0^\infty \int_\R   t^{s-1}  (2\pi \sinh 2t)^{-\frac{1}{2}} 
 f(x)     
\exp\left( -  \frac{(2x-\alpha)^2}{4} \tanh t    - \frac{\alpha^2}{4} \coth t  \right)  dxdt  
\end{multline}
Basic manipulations yield \eqref{dos8fguseoifjsdf}.

\end{proof}

\begin{lemma}
\label{thisisgreallyagreatlemma}
If \(f\in C_u(\R)\), \(\alpha \in \R\) nonzero, then the function 
\(\phi_{f,\alpha} (s) := \Gamma(s)\cdot \trace (fU_\alpha H^{-s}) \), \(\mathrm{Re} (s)>1\),
 extends to an 
analytic function on \(\C\). There are constants \(C_s'\) and \(C_s''\) depending holomorphically on 
\(s\) such that 
\begin{equation}
\label{keyequations@!}
\abs{\phi_{f, \alpha} (s)} \le 
\left( C_s' \alpha^{-2\mathrm{Re}(s)} + C_s'' \right) e^{-\frac{\alpha^2}{4}}\cdot \norm{f}
\end{equation}
for all  \(s\in \C\). 

\end{lemma}

\begin{proof}
The point is that not only does \(\trace (fU_\alpha H^{-s})\) meromorphically extend to \(\C\), 
 the formula \eqref{dos8fguseoifjsdf} defines \(\phi_{f, \alpha} (s)\) by a direct 
integral formula valid for all \(s\in \C\). 

As shown above, a suitable family of states 
\(\mu_{\alpha, t}\) on \(C_u(\R)\), and a constant which we omit, 

\begin{multline}
\label{dos8fguseoif34324jsdf}
\Gamma (s) \cdot \Tr (fU_\alpha H^{-s})  
\sim  \int_0^\infty t^{s-1} \cosech t \exp (-\frac{\alpha^2}{4} \coth t) \, \mu_{\alpha, t} (f) dt
\\=  \int_0^1 t^{s-1} \cosech t \exp (-\frac{\alpha^2}{4} \coth t) \, \mu_{\alpha, t} (f) dt \\ + 
 \int_1^\infty t^{s-1} \cosech t \exp (-\frac{\alpha^2}{4} \coth t) \, \mu_{\alpha, t} (f) dt
=  \zeta_1 (s) + \zeta_2(s).
\end{multline}

Consider first \(\zeta_1(s)\). Since \(\frac{\tanh t}{t}\) and \(\frac{\sinh t}{t}\) are bounded on \([0,1]\), 
we can bound the integrand of \(\zeta_1(s)\) by 
\[ t^{s-2}  e^{-\beta/t}  \, \norm{f}, \;\;\;\; \beta= \frac{\alpha^2}{4}.\]
A change of variables gives 
\[ \int_0^1t^{s-2} e^{-\beta/t} dt = \int_1^\infty t^{-s} e^{-\beta t} dt.\]
If \(A_s = \int_0^\infty t^{-s} e^{-\lambda t}dt \) then \(A_s = \frac{e^{-\beta}}{\beta} + \frac{1}{\beta} A_{s+1}\),
 by integration by parts, and it follows that 
 \(\abs{\int_0^1 t^{s-2} e^{-\beta t} dt} \le \textup{const.} \beta^{-\mathrm{Re}( s)}e^{-\beta}\)
 where the constant does not depend on \(\beta\) or \(s\), and hence that 
 \[ \abs{\zeta_1 (s) } \le C_s'\cdot \norm{f} \cdot \alpha^{-2\mathrm{Re}( s)}e^{-\frac{\alpha^2}{4}}.\]

We can bound \(\zeta_2(s)\) as follows 
\begin{multline}
\label{dos8fguseoifjsdf2}
\abs{ \int_1^\infty t^{s-1} \cosech t \exp (-\frac{\alpha^2}{4} \coth t) \, \mu_{\alpha, t} (f) dt}
\le \norm{f} e^{-\frac{\alpha^2}{4}} \cdot \int_1^\infty t^{s-1} \cosech t dt = C_s'' e^{-\frac{\alpha^2}{4}} \norm{f}
\end{multline}

\end{proof}

The significance of Lemma \ref{thisisgreallyagreatlemma}  is that it 
sheds light on when 
\(\trace (fH^{-s})\) meromorphically extends, when 
 \(a = \sum_{\alpha \in \Gamma} f_\alpha U_\alpha\) is an element of 
 \(C_u(\R)\rtimes \R_d\), and \(\Gamma \subset \R_d\) is a finitely generated
 subgroup. 

The easiest case is that of a cyclic subgroup, and this produces a very strong result. 

If \(f\in C_u(\R)\rtimes \Gamma\) for \(\Gamma \subset \R\) a subgroup, 
let \(f_0\) the coefficient of \(f\) at the identity \(0 \in \Gamma\). Let 
\( \D^\infty\) denote the subspace of \(\D\subset C_u(\R)\) of \(f\) such that 
\(\trace (fH^{-s})\) meromorphically extends to \(\C\) with a simple pole at \(s=1\).

\begin{theorem}
Let \(f\in C_u(\R)\rtimes_\h \Z\), where \(\h \in \R\) is nonzero. 

 Then, if \(f_0\in \D^\infty\), then 
\(\trace (fH^{-s})\), \(\mathrm{Re}(s)>1\), 
 meromorphically extends to \(\C\), with a simple pole at \(s=1\), and 
\[ \Res_{s=1} \trace (fH^{-s}) = \mu_u (f),\]
where \(\mu_u\) is the uniform mean of \(f\) \eqref{restracegeometric}.

\end{theorem}

\begin{proof}

Suppose first that
 \(f\) has expansion \(f = \sum f_n U_{n\h}\) with \(f_0=0\). Then 
\[\Gamma (s) \cdot 
\trace (fH^{-s}) = \sum_n \trace (f_nU_{n\h}H^{-s}) = \sum_n \phi_n (s)\]
where \(\phi_n (s)\) abbreviates \(\phi_{f_n, n\h} (s)\) of Lemma 
\ref{thisisgreallyagreatlemma}. The series converges absolutely and uniformly 
on compact subsets of \(\C\) because of the bound 
\[ \abs{\phi_n (s)} \le \left( C_s' \h^{-2\mathrm{Re}(s)} n^{-2\mathrm{Re}(s)} + C_s''\right) \,
e^{-(\frac{\h^2}{4})n^2}.\] 
due to the Lemma, shows that \(\phi_n\to 0\) exponentially fast as \(n \to \pm \infty\). 

In the general case, \(f = f- f_0 \), \(\trace ((f-f_0)H^{-s})\) extends to an entire function 
for arbitrary \(f\in C_u(\R)\), and \(\trace (f_0H^{-s})\) to a meromorphic function 
with the stated pole structure if \(f_0\in \D^\infty\) by definition, see Theorem 
\ref{restracegeometric} for the equivalent condition to being in \(\D\). 

\end{proof}

\begin{corollary}
\label{corollary:skjfeirjsjkdf}
If \(a\in A_\h := C(\T)\rtimes_\h \Z\), then \(\trace (aH^{-s})\) meromorphically extends to 
\(\C\) with a simple pole at \(s=1\) and residue there \(\tau(a)\), \(\tau\) the standard trace.

Hence the Heisenberg cycle Definition \ref{definition:ds} over the irrational rotation algebra 
\(A_\h\) defines a spectral triple over \(A_\h^\infty\) with the meromorphic continuation 
property over the whole C*-algebra \(A_\h\). 

\end{corollary}

\begin{definition}
\label{definition:diuogroup}
A finitely generated 
subgroup \(\Lambda \subset \R\) with word length function 
\(\abs{\, \cdot \, }_\Gamma\) satisfies a \emph{Diophantine property} if 
\[ \abs{\alpha} \ge C\, \abs{\alpha}_\Gamma^{-\gamma}\]
for some \( \gamma >0\) and \(C>0\). 
\end{definition}

\begin{definition}
Let \(B_\Gamma := C_u(\R)\rtimes \Gamma\), where \(\Gamma \subset \R\) 
is a countable subgroup. Then 
\(B_\Gamma^\infty\) denotes the completion of the (twisted) 
group algebra \(C_u^\infty(\R)[\Gamma]\) with respect to the family of semi-norms 
\( p_{s,m} (f) = \sum_{\alpha \in \Gamma} \norm{f_\gamma^{(m)}} \cdot
 \abs{\gamma}_\Gamma^s.\)
 \end{definition}

 \begin{remark}
  \(B_\Gamma^\infty\) consists of operators in \(C_u(\R)\rtimes \Gamma\) 
 whose expansions
 \( f = \sum_{\alpha \in \Gamma} f_\alpha U_\alpha\) have rapid decay in the 
 sense that
  \begin{equation}
  \label{equation:rapidecay}
  \sum_{\alpha\in \Gamma} \norm{f_\alpha^{(m)} 
  }\cdot \abs{\alpha}_\Gamma^n<\infty,\;\; \forall m,n\ge0.\end{equation}
It is not difficult to prove that \(B_\Gamma^\infty\) is closed under holomorphic functional calculus.
\end{remark}

\begin{theorem}
Suppose \(\Gamma \subset \R\) has a Diophantine property, let 
 \(f\in B_\Gamma^\infty\) and assume \(f_0 \in\D\). 
 Then \(\trace (fH^{-s})\) meromorphically 
extends to \(\C\) with a simple pole at \(s=1\) and 
\(\Res_{s=1} (f) = \mu_u (f)\).

\end{theorem}

\begin{proof}
Write \(f = \sum_{\alpha \in \Gamma} f_\alpha U_\alpha \in B_\Gamma^\infty
\subset \Bound (L^2(\R))\) and assume \(f_0 = 0\).  It suffices to prove that 
\(\trace (fH^{-s})\) extends to an analytic function on \(\C\). 
This equals 
\[\sum_{\alpha \in \Gamma} \trace(  f_\alpha U_\alpha H^{-s}) = 
\sum_{\alpha \in\Gamma} \phi_{f_\alpha, \alpha} (s)\]
where \(\phi_{f, \alpha}\) is notation as in Lemma \ref{thisisgreallyagreatlemma},
and it suffices to 
show that this is an 
absolutely summable sequence of analytic functions, 
uniformly on compact subsets of \(\C\).) 
Shorten notation \(\phi_{\alpha} := \trace ( f_\alpha U_\alpha H^{-s})\). By the 
same Lemma 

\begin{equation}
\label{keyequations@!}
\abs{\phi_{\alpha} (s)} \le 
\left( C_s' \alpha^{-2\mathrm{Re}(s)} + C_s'' \right) e^{-\frac{\abs{\alpha}^2}{4}}\cdot \norm{f_\alpha}
\end{equation}
for all  \(s\in \C\). Since there are potentially infinitely many \(\alpha\) with small absolute value, 
the exponential term is no longer of any use and we disgard it, 
obtaining a polynomial bound for \(\phi_\alpha (s)\) of order 
\(\abs{\alpha}^{\mu}\) for \(\mu  = -2\mathrm{Re} (s) \in \R\).
Since \(\Gamma\) is finitely generated, there exists a
 constant \(C_\Gamma\) such that 
\( \abs{\alpha} \le C_\Gamma\cdot \abs{\alpha}_\Gamma\) for all \(\alpha \in \Gamma\). 
Combining with the Diophantine assumption gives that 
\[  C \abs{\alpha}^{-\gamma}_\Gamma \le \abs{\alpha} \le C' \abs{\alpha}_\Gamma\]
If \(\mu\ge 0 \) we get 
\[ C'' \abs{\alpha}^{-\mu \gamma}_\Gamma \le \abs{\alpha}^\mu 
 \le C''' \abs{\alpha}_\Gamma^{\mu}\]
Hence \[ \sum_{\alpha \in \Gamma} \norm{f_\alpha}\cdot \abs{\alpha}^\mu \le 
C \sum_{\alpha \in \Gamma} \norm{f_\alpha} \cdot \abs{\alpha}_\Gamma^\mu\]
and the last term is finite by  
\eqref{equation:rapidecay}.

If \(\mu<0 \) then we use the bound 
\[ \abs{\alpha}^\mu \le \textup{const.} \abs{\alpha}_\Gamma^{-\mu \gamma}\]
Again, \(\sum_{\alpha \in \Gamma}\norm{f_\alpha} \cdot  \abs{\alpha}_\Gamma^{-\mu \gamma} \) is 
finite by assumption on \(f\) \eqref{equation:rapidecay}.

\end{proof}

\begin{remark}
We make two comments about the proof. 
\begin{itemize}
\item[a)] The Diophantine condition on \(\Gamma\) only seems relevant for the 
zone \( 0 < \mathrm{Re} (s) <1\). 
\item[b)] A weaker condition on \(f\) than 
\eqref{equation:rapidecay} still seems to ensures the result. It suffices to assume that 
\[ \sum_{\alpha \in\Gamma} \norm{f_\alpha}
 \cdot \abs{\alpha}_\Gamma^s \cdot e^{-\frac{\abs{\alpha}^2}{4}}<\infty\]
for all real \(s>0\).  I do not know if such \(f\) form a holomorphically closed subalgebra of \(C_u(\R)\rtimes \Gamma\). 
 However, it does make clear that if \(\Gamma\subset \) is \emph{discrete}, then there is no condition at all 
 on \(f\), except, of course, that \(f\) defines a bounded operator in \(C_u(\R)\rtimes \Gamma\), and that \(f_0 \in \D^\infty\).  
\end{itemize}
\end{remark}

Let \(\alpha\) be a smooth flow on \(M\) compact. Let \(p\in M\) and 
\(\pi_p\colon C(M) \rtimes \R_d \to C_u(\R)\rtimes \R_d\subset \Bound\left( L^2(\R)\right)\) 
be the *-homomorphism induced by restriction of functions to the orbit of \(p\). 
Pulling back the Heisenberg cycle for \(C_u(\R)\rtimes \R_d\) we obtain a cycle for 
\(C(M)\rtimes \R_d\), and for \(C(M)\rtimes \Gamma\) for any \(\Gamma \subset \R\) a subgroup. 

From the results above, if both the flow and the group satisfy Diophantine conditions 
(Definitions \ref{defin:dioflow} and \ref{definition:diuogroup} respectively) then 
the pulled-back cycle determines a spectral triple over a suitable smooth subalgebra of 
\(C(M)\rtimes \Gamma\). However we are mainly interested in applying this to 
the situation of Kr\"onecker flow \(\alpha^\h\) on \(\TT^2\), where both Diophantine 
properties are implied by a Diophantine condition on the number \(\h\).

\begin{theorem}
Let \(B_\h := A_\h \otimes A_{1/\h} \cong C(\TT^2)\rtimes \Gamma\), with 
\(\Gamma \subset \R\) the group generated by \(1, \h\). Let \(B_\h^\infty\) 
be the Schwartz subalgebra of \(B_\h\) of all \(\sum_{\alpha \in \Gamma} f_\alpha U_\alpha\) 
with 
\[ \sum_{\alpha \in \Gamma} \norm{X^nf} \cdot \abs{\alpha}_\Gamma^m<\infty.\]
where \(X\) generates the flow. 

Then the 
Heisenberg bi-cycle of Definition \ref{bicycle} determines a spectral triple over 
\(B_\h^\infty\) with the meromorphic extension property.

\end{theorem}

\section{Topology of Heisenberg cycles}

We now use cyclic cohomology to perform some K-theory and index-pairing 
calculations with the Heisenberg cycles over \(A_\h\) and \(B_\h = A_\h \otimes A_{1/\h}\).

We will focus on \(A_\h\). Let \([D_\h] \in \KK_0(A_\h, \C)\) be the Heisenberg class 
of Definition \ref{definition:ds}, involving 
\begin{itemize}
\item[a)] The unitary action of \(\Z\) on \(L^2(\R)\) with \(n\) acting by \(U_\h^n\). 
\item[b)] The action of \(C(\T) = C(\R/\Z) = C(\R)^\Z\) of \(\Z\)-periodic functions by multiplication operators. 
\item[c)] The operator \(D = \begin{bmatrix} 0 & x-\d/dx\\x+d/dx & 0  \end{bmatrix}\)
\end{itemize}
From a) and b) we get the representation \(\pi_\h\colon A_\h \to \Bound
\left(L^2(\R)\right)\)

It is important to note that \(A_\h\) only depends, of course, on the class of \(\h\) mod \(\Z\), 
but \(\pi_\h\) depends on \(\h\), not just its equivalence class. (This is a difference between 
our set-up and that of \cite{LM}, as they only allow \(0 <\h <1\).) 

 In particular, 
fixing \(0 < \h <1\), then the \([D_{\h+b}]\), for \(b\in \Z\), define a \(\Z\)-family of 
spectral cycles over the single \(A_\h\). 

In fact, due to the identity \([\Hilb_\h]\otimes_{A_\h} [\dol] = [D_\h]\), this is accounted 
for by a similar fact about the Morita equivalences \(\Hilb_\h\) of Theorem 
\ref{theorem:reiffelsmeb}. Namely, that for \(\h\) varying within a \(\Z\)-coset of \(\R\), 
we obtain a \(\Z\)-parameterized family of Morita 
\(A_{1/\h}\)-\(A_\h\)-bimodules: note that \(A_{1/\h}\) changes as an integer is added to 
\(\h\), and so does the bimodule, 
but the ring of scalars \(A_{\h}\) in the right multiplication 
remains the same. Twisting \(\dol\) by these bimodules and forgetting the 
left action gives a \(\Z\)-parameterized family of spectral triples these are exactly 
the \([D_{\h +b}]\).

Actually, there is a certain internal symmetry of \(A_\h\), which we call the 
`Heisenberg twist' (which appeared already in \cite{DE}) 
whose iterates act transitively on these sets of data in both 
the K-theory and K-homology picture. 
 The Heisenberg 
twist is a KK-morphism determined by the \emph{bundle} of cycles \(D_\h\), over 
\(\R\), as we explain below.

We are going to use the 
Local Index Theorem of
 \cite{CM:Local} in dimension \(2\) to compute the pairing of the 
 class \([D_\h]\) with \(\K_0(A_\h)\); the result, and its proof, effectively 
 computes the index theory of the class \(\Delta_\h\) as well. 
 We will use the development of the 
 Local Index Formula by N. Higson in \cite{Higson:Local} and our 
 results on the harmonic oscillator 
 residue trace of the previous section.

\begin{remark}
 The index formula Corollary \ref{corollary:indextheorem} 
  we arrive at is very similar to one
 due to Connes (see \cite{Connes}, \cite{Connes:note}),
 who discovered it using his 
 computation of the 
 cyclic theory of \(A_\h^\infty\). There are some minor differences, 
 as Connes fixes \( 0 <\h <1\) and 
 computes the index map \(\K_0(A_{1/\h}) \to \Z\) induced by 
 \(D_\h\), and we are allowing arbitrary \(\h\) and computing 
 the index pairing with \(\K_0(A_{\h})\). Connes' used some 
 basic results about the cyclic cohomology of \(A_\h\) to 
 deduce his theorem. Our method uses the 
 Local Index Theorem and is more general. 
 
\end{remark}

In the paper \cite{Connes:note}, Alain Connes described an invariant of a finitely 
generated projective module over \(A_\h\), generalizing the first Chern number of a 
complex vector bundle over \(\TT^2\). 

Connes' construction was the following. 
Let \(A\) be any C*-algebra endowed with an action of \(\R^2\) by 
automorphisms with \((s,t)\) acting by \(\alpha_s\circ  \beta_t \).

Let \(\delta_i \colon A\to A\) be the densely defined derivations 
\[ \delta_1 (a) := \lim_{t\to 0} \frac{\alpha_t (a) - a}{t}, \;\;\;\;\delta_2 (a) := 
\lim_{t\to 0} \frac{\beta_t (a) - a}{t}, \;\;\;\;\; a\in A^\infty,\]
where \(A^\infty = \cap_{n,m} \dom (\delta^n)\cap \dom(\delta^m),\)
the *-subalgebra of elements such that \( (s,t) \mapsto \alpha_s(\beta_t(a))\) is 
smooth. 

In addition, let \(\tau \colon A \to \C\) be an \(\R^2\)-invariant tracial state. 

Then 
Connes' invariant of a f.g.p. module \(eA^\infty\), where \(e\) is a projection in \(A^\infty\), 
 is given by 
\[ c_1 (e) := \frac{1}{2\pi i} \tau \left( e [\delta_1 (e), \delta_2(e)]\right).\]
We call 
\(c_1(e)\) the \emph{first Chern number of \(e\)}. 

The number \[ c_1(e)  := \frac{1}{2\pi i} \tau \left( e \,[\delta_1 (e), \delta_2(e)]\right)\] only 
depends on the equivalence class of \(e\) in \(\K_0(A^\infty)\) (see\cite{Connes:survey}).

Moreover, 
\(c_1(\Hilb \oplus \Hilb') = c_1(\Hilb) + c_1(\Hilb')\) and \(c_1\) thus determines a group 
homomorphism 
\(\K_0(A) \to \R\).

Let \(\h \in \R\) and \(A_\h = C(\TT)\rtimes_h \Z\) the corresponding rotation algebra, 
with \(u \in A_\h\) the generator of the \(\Z\)-action. 
Then the \(\R^2\)-action with \(\alpha_t (f) = f(x-t) \), \(\alpha_t (u) = u\);  
\(\beta_t (u^n) =  e^{2\pi i nt} u^n\),  \(\beta_t (f) = f\), 
 gives rise to the derivations 
  \[\delta_1 (\sum_n f_n [n]) = \sum_n f_n'\, [n], \;\;\;\;\;
\delta_2 (\sum_n f_n[n]) = \sum_n  2\pi in \cdot f_n [n].\]

Recall the Rieffel modules described in Lemma 
\ref{theorem:reiffelsmeb}. For any \(\h\), \(\Hilb_\h\) is the completion of 
\(C_c(\R)\) with respect to the \(A_\h\)-valued inner product 
 \[ \langle \xi, \eta \rangle_{A_\h} (x, m) = 
 \sum_{n \in \Z}  \overline{\xi ( x-n)} \eta ( x-n-m\h), x\in \R/\Z, \;\; m \in \Z\]  
 Give \(C_c(\R)\) the right \( C(\R/\Z)\rtimes_\h \Z = A_\h\) module structure 
 with 
 \[ (\xi n ) (x) = \xi (x+n\h), \;\; ( \xi f) (x) = f(x) \xi (x).\]
 This extends to a multiplication \(\Hilb_\h \times A_\h \to \Hilb_\h\) giving a
  finitely generated and projective Hilbert \(A_\h\)-module. To find 
 a projection \(p_\h\) such that \(p_\h A_\h \cong \Hilb_\h\), recall that the 
 \(A_{1/\h}\)-valued inner product is given by 
\[ _{C(\R/\h \Z) \rtimes \Z} \langle \xi, \eta\rangle ( x, m) = 
 \sum_{n \in \Z} \xi ( x-n\h) \cdot \overline{\eta (x-n\h -m)}, \;\; x\in \R/\h \Z, \;\; m \in \Z,\]
 If we can find \(\xi\in \Hilb_\h\) such that \(_{C(\R/\h \Z) \rtimes \Z} \langle \xi, \xi\rangle = 1\), 
 then the map \(\eta \mapsto \langle \eta , \xi\rangle_{A_\h}\) embeds 
 \(\Hilb_\h\) in the trivial rank-one \(A_\h\)-module, and the required projection is 
 \(\langle \xi, \xi\rangle_{A_\h}\). We thus seek \(\xi\) such that 
 \[   \sum_{n \in \Z} \xi ( x-n\h) \cdot \overline{\xi (x-n\h -m)}  
 = \delta_{0m}\]
 in Kr\"onecker notation. In particular, \(\xi\) should have support 
 contained within an interval, e.g. \([0,1]\), of length \(1\), for the outcome to be 
 zero if \(m\not=0\). But this means that if \(\h >1\) the function 
 \( \sum_{n \in \Z} \xi (x-n\h)^2\) will have zeros.  A necessary condition 
 therefore to find such \(p_\h\) is that \( 0 <\h <1\).  For larger \(\h\), 
 one must find a finite set \(\xi_1, \ldots , \xi_m\) of vectors in \(\Hilb_\h\), and 
 use them to embed \(\Hilb_\h\) in \(A_\h^m\). One obtains projections 
 \(p_\h\) in matrix algebras over \(A_\h\). This makes the computation 
 of curvature more complicated. 
 
 If \(0 < \h <1\) then the problem can be solved: the ensuing projection 
 is of the form 
\[ p_\h = f + gu +  g^{-\h} u^*\]
where \(f\) and \(g\) are suitably chosen functions. For 
\(a\in (0,1)\), and \(\epsilon>0\) small, 
\(f\) equals zero on \( [0,a]\) and on \([ a+\h + \epsilon, 2\pi]\), 
 and \(f=1\) on 
\([a+\epsilon, a+\h]\). We choose \(f\) so that 
\( f(x) + f(x+\h) =1\). We set 
\(g = \sqrt{f-f^2}\) on 
\([ a+\h, a + \h + \epsilon]\) and is zero otherwise.

The following calculation from \cite{Connes:survey} is reproduced 
below for the benefit of the 
reader. 

\begin{lemma} 

Let \(p_\h \in A_\h = C(\TT)\rtimes_\h \Z\) be the Rieffel projection 
Then \(c_1 (p_\h) = +1.\)

\end{lemma}

\begin{proof}
For brevity, for \(f\in C(\TT)\) understood as a 
\(\Z\)-periodic function on \(\R\), let \(f^\h (x) := f(x-\h)\) 
denote the action. 

The Rieffel projection is given by
\( p_\h = f + gu +  g^{-\h} u^*\) as above. 
 We then have to compute \(c_1(e) = \frac{1}{2\pi i } \tau ([ \delta_1 (p_\h), \delta_2(p_\h)]\). 
 We first compute 
 \begin{multline}
 \frac{1}{2\pi i }[ \delta_1 (p_\h), \delta_2(p_\h)]
 = u^* ( gf' - gf'^\h) + 2\left( (gg')^{-\h} -gg' \right) + (gf'-gf'^\h) u
 \end{multline}
 Multiplying this on the left by \(p_\h\) produces a terrific mess, but we are only interested 
 in its trace, so the only part which is relevant is 
 \begin{equation}
 g^2(f'-f'^\h) + 2f\left( (gg')^{-\h} -gg'\right) + \left( g^2(f'-f'^\h)\right)^{-\h}
  \end{equation}
 which we want to integrate over \(\TT\). 
 
 Set \(w = g^2, v= f-f^\h\). The integral is given by 
 \begin{equation}
\int   wv'  + f\left(   w'^{-\h} -w'\right) + (wv')^{-\h}
  \end{equation}
 
 The middle term is 
 \[ \int fw'^{-\h} - \int fw' = \int f^\h w'- \int fw'= -\int vw'.\]
 Hence we are reduced to computing 
 \[ \int wv' -vw' + (wv')^{-h}  = \int 2wv' - vw' = 3\int wv',\]
 by integration by parts. Next, since \(f^\h = 1-f\) on \(\supp (g)\), 
we can replace  \( f' - f'^\h \) by \(-2f'\) and get 
 \begin{multline}
  -6\int   f' (f-f^2) = -6\int f'f +6\int f'f^2 \\ = -3\int (f^2)' + 2\int (f^3)' 
  = 3-2=1,
 \end{multline}
where the integration is understood to be restricted to the support of \(g\). This completes the calculation. 

\end{proof}

The first Chern class of \emph{any} projection in \(A_\h^\infty\) is an integer, a fact 
related to the Quantum Hall Effect (see \cite{Connes}).  This is 
due to agreement of the first Chern number with the 
index pairing (an integer) of
 the projection and the Dirac-Dolbeault spectral 
cycle. This a consequence of the 
Local Index Formula of Connes and Moscovici, which we are going to 
work out in the case of the Heisenberg cycles, but first state in low dimensions.

\begin{theorem} (Connes-Moscovici, \cite{CM:Local}.) 
\label{theorem:localindexhteorem}
Let \( (H, \pi, D)\) be an even,  \(2\)-dimensional spectral triple over \(A^\infty \subset A\), 
for a 
C*-algebra \(A\), regular and with the meromorphic continuation property over 
\(A^\infty\).

Let \([D]\in \KK_0(A, \C)\) be the class of the triple. 
 Let \(\Delta := D^2\), and let \(\epsilon\) be the grading operator on \(H\). 

Define functionals 
\begin{itemize}
\item[a)]  \(\psi_0 \colon A^\infty \to \C\), 
\[ \psi_0 (a) := \Res_{s=0}\;  \Gamma (s) \cdot \Tr (\epsilon a (\Delta + \proj_{\ker D} )^{-s}) ,\]
and
\item[b)]  \( \psi_2 \colon A^\infty\otimes A^\infty\otimes A^\infty  \to \C, \), 
\[ \psi_2 (a^0, a^1, a^2) := \frac{1}{2} \Res_{s=1} ( \epsilon a^0 [D, a^1][D, a^2]\Delta^{-s}).\]
\end{itemize} 

Then, if  \(e\in A^\infty\) is a projection, then 
\[ \langle [e], [D]\rangle = \Psi_0 (e) - \Psi_2(e-\frac{1}{2}, e, e),\]
where 
\( \langle [e], [D]\rangle\in \Z\) is the pairing between the \(\K_0(A)\)-class 
\([e]\) and the \(\KK_0(A, \C)\) class \([D]\). 
\end{theorem}

The following Lemma shows that zero-dimensional 
 part of the Chern character of 
a Heisenberg cycle comes from taking the pole at \(s=1\) of 
the zeta function \(\trace (aH^{-s})\) discussed 
in the previous section. 

\begin{lemma}
\label{lemamnotreallytrnew}
If \(a\in \Bound (L^2(\R)\) then 
\[\Gamma (s)  \Tr (\epsilon a (\Delta + \proj_{\ker D} )^{-s}) -  \trace (aH^{-s-1})\]
extends analytically to \(\mathrm{Re} (s) >-1\). In particular, 
if \(\Psi_0\) is the functional b) in Theorem \ref{theorem:localindexhteorem}, and if \(\trace (aH^{-s})\) 
meromorphically extends to \(\C\) then 
\[ \Psi_0 (a) = \restrace (a):= \Res_{s=1} \trace (aH^{-s}).\]

\end{lemma}

\begin{proof}

We refer to Theorem \ref{theorem:localindexhteorem}. The Hilbert space for the 
Heisenberg triple is the direct sum of two copies of \(L^2(\R)\), and 
\( D = \begin{bmatrix} 0 & x-d/dx\\ x+d/dx& 0 \end{bmatrix}\). The kernel of 
\(D\) is the same as the kernel of \(x+d/dx\), and is spanned by the 
ground state \(\psi_0 (x) = \pi^{-\frac{1}{4}}e^{-\frac{x^2}{2}}\), and \(\Delta = 
\begin{bmatrix} H - 1  & 0 \\ 0 & H+1\end{bmatrix}\) where \(H\) is the harmonic 
oscillator. If \( p = \proj_{\ker D}\) then 
\( \Delta + p  = 
\begin{bmatrix} H - 1 + p  & 0 \\ 0 & H+1\end{bmatrix}.\)
The first copy of the Hilbert space \(L^2(\R)\) is even in the grading, the second is odd
and so the meromorphic function whose pole at \(s=0\) gives \(\Psi_0(a)\) 
\begin{multline}
 \Gamma (s)    \trace ( \epsilon a(\Delta + p)^{-s}) = 
\Gamma (s) \cdot \Tr (a (H - 1 + p )^{-s})  - \Gamma (s) \cdot \Tr(a (H+1)^{-s}).\end{multline}
for \(a \in C_u(\R)[\R_d]\). 
Applying Mellin transform and small calculation gives 
\[ =  \int_0^\infty t^{s-1} \sinh t \cdot \trace (ae^{-tH}) dt +  E (s)  \]
 where 
\[ E(s) = \int_0^\infty t^{s-1} e^t \cdot \trace ( e^{-tH} - e^{-t(H+p)} ) dt
= \trace (ap) \cdot \int_0^\infty t^{s-1} (1-e^{-t}) dt,\]
giving that \(E(s)\) extends meromorphically with a simple pole at 
\(s = -1\), and in particular, \(E(s) \) is analytic for \(\mathrm{Re} (s) >-1\). 
Hence 
\[ \Psi_0 (a) = \Res_{s=0}\; \int_0^\infty t^{s-1} \sinh t  \, \trace (ae^{-tH}) dt.\]
Since \(\sinh t \sim t\) as \(t \to 0\), 
\begin{multline}
\Psi_0 (a) =  \Res_{s=0} \int_0^\infty t^{s} \trace (ae^{-tH}) dt = 
\Res_{s=1}\int_0^\infty t^{s-1} \trace (ae^{-tH}) dt \\= \Res_{s=1}\trace (aH^{-s}).\end{multline}

\end{proof}

Let \(B = C(M)\rtimes \Lambda\), for 
a flow \(\alpha\) on \(M\) and \(\Lambda \subset \R\) a Diophantine 
subgroup. Let \(\mu\) be an \(\alpha\)-invariant measure, \(X\) generate the flow, and 
assume that \(Xu= f\) is smoothly solvable for any 
smooth \(f\) such that \(\int_M fd\mu = 0\).
 Fix \(p\in M\) and let \(\pi\colon C(M) \rtimes \Lambda \to \Bound (L^2(\R))\) 
the corresponding representation, with \(\pi (f) = f_p\), \(f_p (t) = f(\alpha_t(p))\). 

We have shown that \(B^\infty \subset B\) has the property that 
\( \trace ( \pi (b) H^{-s})\) has the meromorphic extension property and that 
\[ \Res_{s=1} ( \pi(b)H^{-s}) = \tau_\mu (b),\]
where \(\tau_\mu\colon B\to \C\) is the trace induce by \(\mu\), and \(b\in B^\infty\).

Let \(\delta^\alpha_1, \delta^\alpha_2\) be the derivations of \(B^\infty\) defined 
\[ \delta_1^\alpha (f) = X(f),\; \delta^\alpha_1 ( U_\alpha) = 0 , 
\;\;\;\;\;\;\;\;\;\;  \delta^\alpha_2 (f) = 0, \;  \delta_2^\alpha (U_\alpha) = \alpha.\]

\begin{lemma}
In the above notation, the functional \(\Psi_2\) of Theorem 
\ref{theorem:localindexhteorem} b) is given on \(B^\infty\) by 
\begin{equation}
\Psi_2(b^0,b^1,b^2) 
=  \tau_\mu \left( a^0\delta^\alpha_1(a^1)\delta^\alpha_2(a^2) - a^0 \delta^\alpha_2 (a^1)
\delta_1^\alpha(a^2)\right) 
\end{equation}
for all \(b_0, b_1, b_2\in B^\infty\). 

\end{lemma}

\begin{proof}
Note that  
\[ \delta_1 (b) = [ \pi (b), \frac{d}{dx}],\;\;\;\;\;\;\delta_2 (b) = [\pi (b), x].\]
Expand $[D,a^1][D,a^2]$ as a block matrix 
\begin{align*}
\left[D,b^1\right]\left[D,b^2\right] &= \left[\begin{array}{cc}
\left[x-d/dx,b^1\right]\left[x+d/dx,b^2\right]&0\\
0&\left[x+d/dx,b^1\right]\left[x-d/dx,b^2\right]
\end{array}\right]
\end{align*}
We deduce 
 \begin{multline}
 \frac{1}{2}\restrace\left(\ve b^0\left[D,b^1\right]\left[D,b^2\right]\Delta^{-1}\right)\\
=\restrace\left(b^0\left[x,b^1\right]\left[d/dx,a^2\right]\right) -
 \restrace\left(b^0\left[d/dx,b^1\right]\left[x,b^2\right]\right) 
\\=  \tau_\mu \left( b^0\delta_1(b^1)\delta_2(b^2) - b^0 \delta_2 (b^1)\delta_1(b^2)\right) 
\end{multline}

\end{proof}

We have thus proved the following.

\begin{theorem}
Let \(B = C(M)\rtimes \Lambda\), for 
a smooth flow \(\alpha\) on a compact manifold 
\(M\). Let 
\(\Lambda \subset \R\) a Diophantine 
subgroup. Let \(\mu\) be an \(\alpha\)-invariant measure, \(X\) generate the flow, and 
assume that \(Xu= f\) is smoothly solvable for any 
smooth \(f\) such that \(\int_M fd\mu = 0\).
 
 Then the Chern character of the Heisenberg cycle determined by a point \(p\in M\) is 
 given by \(\tau_\mu - \tau_2^\alpha\), where \(\tau_\mu \) is the trace on \(B\) determined by 
 \(\mu\), and \(\tau_2^\alpha\) is the cyclic \(2\)-cocyle
 \[ \tau_2^\alpha(b^0,b^1,b^2) = 
 \tau_\mu \left( b^0\delta^\alpha_1(b^1)\delta^\alpha_2(b^2) - b^0 \delta^\alpha_2 (b^1)
\delta_1^\alpha(b^2)\right) \]
on \(B^\infty\).

\end{theorem}

\begin{corollary}
\label{theorem:bigone}
Let \(\h \in \R\) and \(A_\h := C(\TT)\rtimes_\h \Z\) the corresponding rotation algebra. 

Let \([D_\h]\) be the class of the Heisenberg cycle (Definition \ref{definition:ds}) 
\[\left(L^2(\R) \oplus L^2(\R), \pi_\h, \; D:= \begin{bmatrix} 0 & x-d/dx  \\
x+d/dx & 0 \end{bmatrix} \right) .\]

Then the Chern character of \([D_\h]\) is given by 
\(\tau - \h \tau_2\), where \[\tau_2 (a^0,a^1,a^2) = 
 \tau \left( a^0  \delta_1 (a^1)\delta_2(a^2) - 
a^0 \delta_2 (a^1) \delta_1 (a^2)\right),\]
the curvature cocycle of Connes.

\end{corollary}

\begin{corollary}
\label{corollary:indextheorem}
Let \( e\in A_\h^\infty\) be a projection, \([e]\in \K_0(A_\h)\) its class. Then 
\[ \langle [e], [D_\h]\rangle = \tau (e) - \h  c_1(e),\]
where 
\(c_1(e) \) is its first Chern number.

In particular, if \(p_\h\) is the Rieffel projection for \(\h\) then 
\[ \langle [p_\h] , [D_h]\rangle = - \left \lfloor{\h}\right \rfloor \]
where 
\( \left \lfloor{\h}\right \rfloor \) is the greatest integer \(<\h\).

\end{corollary}

The proof is immediate from the formula \ref{theorem:bigone}, because 
\(c_1(p_\h) = 1\) and \(\tau (p_\h) = \h- \left \lfloor{\h} \right\rfloor\).

This yields the proof of 
Proposition \ref{indexpaiinrgs}.

We close with a remark.  The integrality of 
Connes' first Chern number is due to the index result that 
\[ \langle [\dol] , [e]\rangle = c_1(e),\]
where \(c_1(e)\) is the first Chern number of \(e\). If one combines 
this with our Index computation for \([D_\h]\) then we obtain the 
following `gap-labelling' result. It is of course well-known; but 
the argument below does not depend on computation of \(\K_0(A_\h)\).

\begin{corollary}
\label{corollary:rangeofrangeforateha}
Suppose \(\h \in \R\) is nonzero. Then 
if \(\tau \colon A_\h \to \C\) is the trace, \[\tau_*\colon \K_0(A_\h) \to \R\] the induced
group homomorphism, then then range of 
\( \tau_*\left( \K_0(A_\h)\right)\) is the subgroup \(  \Z +  \h \Z\subset \R .\)

\end{corollary}

\begin{proof}
If \(e \in A_\h^\infty\) is a projection, then application of our results above gives that 
\( \tau (e) + \h \cdot c_1(e)\) is an integer. On the other hand, \(c_1(e)\) is an integer. This 
implies \(\tau (e) = m + n\h\) for a pair of integers \(m, n\). Finally, \(A^\infty_\h\) is 
dense and holomorphically closed in \(A_\h\), so any projection in 
\(A_\h\) is represented by a projection in \(A_\h^\infty\).

\end{proof}

The interest in this argument is that it computes the range of the trace using 
a combination of two spectral cycles, but does not rely on 
computation of \(\K_*(A_\h)\). It is conceivable that such a method could be applied 
in connection with other `Gap-Labelling' problems.

\section{Transverse foliations}

There is a well-known procedure for producing finitely generated projective 
(f.g.p.) modules over \(A_\h\), using Morita equivalence. Fixing one of the 
standard linear loops in \(\TT^2\) determines a Morita equivalence of 
\(A_\h\) with the C*-algebra \(B_\h = C(\TT^2)\rtimes_\h \R\) of the 
Kronecker foliation \(\mathcal{F}_\h\) 
into lines of slope \(\h\). On the other hand, \emph{any} 
linear loop in \(\TT^2\) is transverse to \(\mathcal{F}_\h\). These linear 
loops, parameterized by pairs of relatively prime integers, determine 
therefore Morita equivalences between \(A_\h\) and what turn out to 
be other rotation algebras, and in particular, unital algebras. Hence 
they determine f.g.p. modules over \(A_\h\); these
parameterize (the positive part of) the K-theory. 

In \cite{DE} we showed that there is an analogue of this procedure using 
non-compact transversals: if \(\h'\not= \h\) then the Kronecker 
foliations \(\mathcal{F}_\h\) and \(\mathcal{F}_{\h'}\) are transverse. 
Their product foliates \(\TT^2\times \TT^2\) and its restriction to the 
diagonal gives an equivalent \'etale groupoid. This reasoning produces 
for every \(\h'\not= \h\) a f.g.p. module \(\LL_{\h, \h'}\) over \(A_\h\otimes A_{\h'}\) 
and corresponding \(\K\)-theory class \([\LL_{\h, \h'}]\in \KK_0(\C, A_\h, \otimes A_{\h'})\). 
In particular, fixing $\hbar' = \hbar + b$ for any integer \(b\not= 0\) gives a f.g.p. module over \(A_\h \otimes A_\h\). 
We denote it \(\LL_b\).

Let 
\[\PD\colon \KK_0(\C , A_\h\otimes A_\h) \to \KK_0(A_\h, A_h),\]
be Connes' Poincar\'e duality map \cite{Connes}. In \cite{DE} it is proved that 
\begin{equation}
\label{equation:pdforatheta}
\PD ([\LL_b]) = \tau_b,
\end{equation}
where \(\tau_b\) is the Kasparov morphism defined in terms of Dirac-Schr\"odinger 
operators as follows. We take the standard right Hilbert \(A_\h\)-module 
\(L^2(\R) \otimes A_\h\). We let \(\Z\) act on the left by the following formula, where 
we designate a dense set of elements of our Hilbert module in the form
\( \sum_{n \in \Z} \xi_n \cdot [n]\), with \(\xi_n \in L^2(\R)\otimes C(\TT)\): 

\[ k \cdot \left( \sum_{n \in \Z} \xi_n \cdot [n] \right) := \sum_{n \in \Z} k(\xi_n) \cdot [k+n].\] 
where
 \(k(\xi) ([x], t) = \xi ( [x-k\h], t-k)\), 
 \(\xi \in L^2(\R)\otimes C(\TT)\). 

Let \(f\in C(\TT)\) act by 
\[ f\cdot \left( \sum_{n \in \Z} \xi_n \cdot [n] \right) 
:= \sum_{n \in \Z} f^b\cdot \xi_n \cdot [n].\] 
where \(f_b (t, [x]) = f([x+tb]).\) 

These two assignments determine a covariant pair and 
representation \[\pi_\h\colon A_\h \to 
\Bound (L^2(\R)\otimes A_\h).\]

\begin{definition}
The \emph{Heisenberg twist} \(\tau_b\in \KK_0(A_\h, A_h)\) is the class of the 
spectral cycle 
\[ \left( L^2(\R)\otimes A_\h\, \oplus \, L^2(\R) \otimes A_\h, \, \pi_h \oplus \pi_h, 
D\otimes 1_{A_\h}\right),\]
with \( D:= \begin{bmatrix} 0 & x-d/dx  \\
x+d/dx & 0 \end{bmatrix}  \). 

\end{definition}

The equality \eqref{equation:pdforatheta} gives rise to an explicit geometric cycle 
representing the unit in Connes' duality \cite{DE}. The relationship between the
 Heisenberg twist and the Heisenberg cycles is implied by the following 
 Lemma.

\begin{lemma}
Let \(\h, \mu \in \R\).
\begin{itemize}
\item[a)] The group multiplication \(m \colon \TT\times \TT\to \TT\) intertwines the 
diagonal \(\Z\)-action on \(\TT\times \TT\) by group addition of \( (\h, \mu)\) and 
group addition on \(\TT\) of \(\h + \mu\), and so determines a *-homomorphism 
\[ \chi   \colon A_{\h + \mu} \to A_\h \otimes A_\mu.\]
In this notation: 
 \[\chi ^*([D_\h]\otimes_\C [D_\mu]) = [D_{\h + \mu}].\]
\item[b)] If \(\mu = b\in \Z\), then 
\[\tau^b = \chi^*([D_b]\otimes 1_{A_\h}) \]
 where \(\tau^b\) is the 
Heisenberg twist. 
\item[c)] \(\tau^b\otimes_{A\h} [D_\h] = [D_{\h + b}]\) for any \(\h\in \R, b\in \Z\). 
\end{itemize}

\end{lemma}

\begin{proof}
For c) we have 
\[ \tau^b\otimes_{A\h} [D_\h]) = 
 \chi^*([D_b]\otimes 1_{A_\h} )\otimes_{A_\h} [D_\h] = 
 \chi^*( [D_b]\otimes_\C [D_h]) = [D_{b+\h}]\]
 using first part b) and then part a).

 b) follows from an inspection at the level of cycles: they differ 
 only in the representations, which are clearly homotopic.

We now prove a). 

Consider 
\(\chi ^*([D_\h]\otimes_\C [D_\mu]),\)
a class in $\KK(A_{\theta+\eta},\C)$. By the standard 
method of computing external products, it 
is represented by the following spectral cycle. 
The Hilbert space is 
module is $L^2(\R,\C^2)\otimes L^2(\R,\C^2)$ 
and operator $D\otimes 1 + 1\otimes D$. With \(u,v \in A_\h\) the 
standard unitary generators, \(u = z, v = [1] \), 
the representation is given by:
$$(u\cdot \phi)(x,y)= e^{2\pi i (x+y)}\phi(x,y),\quad (v\cdot \phi)(x,y) = \phi(x-\theta,y-\eta)$$
We apply a homotopy to the representation, with $t\in[0,1/2]$:
$$(\pi_t(v)\cdot\phi)(x,y) = \phi(x - (1-t)\theta - t\eta, y - t\theta - (1-t)\eta)$$
The resulting cycle is $(L^2(\R,\C^2)\otimes L^2(\R,\C^2), \Delta, D\otimes 1 + 1\otimes D)$, where $\Delta$ is the ``diagonal" representation of $A_{\theta+\eta}$:
$$(u\phi)(x,y) = e^{2\pi i(x+y)}\phi(x,y),\quad (v\phi)(x,y) = \phi\left(x-\frac{\theta+\eta}{2}, y-\frac{\theta+\eta}{2}\right)$$\vspace{5pt}

Next consider the class \( [D_{\h + \mu}]\). 
The the unit in $1_\C \in \KK(\C,\C)$ can be represented by the cycle $(L^2(\R,\C^2), 1, D)$. Taking the intersection product of this with $[D_{\theta+\eta}]$ yields a cycle which is equivalent to $[D_{\theta+\eta}]$, but more closely resembles the cycle described in the previous 
paragraph: 
the Hilbert space
 is $L^2(\R,\C^2)\otimes L^2(\R,\C^2)$, the operator is $D\otimes 1 + 1\otimes D$, 
 and the representation is given by:
$$(u\cdot\phi)(x,y) = e^{2\pi i x}\phi(x,y),\quad (v\cdot \phi)(x,y) = \phi(x-\theta-\eta,y)$$
From here we take a homotopy by rotating this representation around $\R^2$ to lie along the diagonal (i.e. so that $(a\cdot\phi)(x,y)$ depends only on $x+y$), and the result follows. 
 
By b)
the equation $\tau_b\otimes_{A_\h}  [D_\h] = [D_{\h+b}]$ follows immediately.

\end{proof}

\begin{corollary}
The Heisenberg twist acts by the identity on \(\K_1(A_\h)\). 

With respect to the ordered basis \( \{ [1], [p_\h]\}\)
 for \(\K_0(A_\h)\), where \(p_\h\) is the Rieffel projection, the 
 morphism \(\tau_b\) acts 
by matrix multiplication by \( \begin{bmatrix} 1 & b \\ 0 & 1\end{bmatrix}.\)

\end{corollary}

\begin{proof}
The first statement follows from \cite{DE}. 

Consider 
\( (\tau^b)_*([p_h]) \in \K_0(A_\h)\). Write 
\[ (\tau^b)_*([p_h]) = x[1] + y[p_\h].\]
Pairing both sides with \([D_h]\) and using the Index Theorem 
Corollary \ref{corollary:indextheorem} twice gives 
\[x =  (\tau^b)_*([p_h])\otimes_{A_\h} [D_\h] = [p_\h] \otimes_{A_\h} (\tau^b)^*([D_h]) = 
[p_\h]\otimes_{A_\h} [D_{h+b}] = b.\]
Pairing the same equation with \([D_{\h + 1}]\) and computing give that 
\(y = 1\).

\end{proof}

It follows from similar simple arguments that
in the classical case \(\h  = b \in \Z\), the Heisenberg 
classes \([D_b]\in \KK_0(C(\TT^2), \C) = \K_0(\TT^2)\) 
are given by 
\[ [D_b] = [\pnt] + b\cdot [\dol]\;\;\in \K_0(\TT^2).\]
For \(b=1\), we have noted that \([D_1] =  [\dol \cdot \Poincare]\). 
This corresponds to \([\dol \cdot \Poincare] = [\pnt] + [\dol]\), which 
of course follows from the Riemann-Roch formula. 
We have 
\[ \langle [D_n], [E]\rangle = \dim E + n\cdot c_1 (E),\]
for any complex vector bundle \(E\) over \(\TT^2\). Therefore 
the classes  \([D_n]\) taken together determine both the the dimension 
and first Chern number, the two basic invariants of a 
complex vector bundle over \(\TT^2\).


\begin{thebibliography}{00}



\bibitem{Berline} N. Berline, E. Getzler, M. Vergne, \emph{Heat kernels and Dirac operators}, 
Grundlehren der mathematicischen Wissenshcaften (1992), Springer-Verlag Hedelberg NewYork. 


\bibitem{Connes:note} A. Connes: \emph{C*-algebras et g\'eomerie differentielle}. 
C.R. Acad. Sci. Paris, Ser. A-B  (1980), 599-604. 


\bibitem{Connes} A. Connes: \emph{Noncommutative Geometry}, Academic Press (1994).

\bibitem{Connes:survey} A. Connes: \emph{A survey of foliations and operator algebras,} Operator algebras and applications, Part I (Kingston, Ont., 1980), Proc. Sympos. Pure Math., 38, Amer. Math. Soc., Providence, R.I., (1982), 521–628.

\bibitem{CM:Local} A. Connes, H. Moscovici: 
 \emph{The local index formula in noncommutative geometry}, 
Geom. Funct. Anal. 5 (1995), no. 2, 174–243.


\bibitem{DE} A. Duwenig, H. Emerson: \emph{Transversals and duality for the irrational 
rotation algebra}. Trans. AMS (accepted 2020), To appear. 


\bibitem{Emerson:Book} H. Emerson: \emph{An introduction to C*-algebras and Noncommutative Geometry}. Book in preparation. 


\bibitem{Forni1}G. Forni: \emph{ Sobolev regularity of solutions of the cohomological equation},
Ergodic Theory Dynam. Systems 41 (2021), no. 3, 685–789. 

\bibitem{Forni2}  L. Flaminio, G. Forni: \emph{Invariant distributions and time averages for horocycle flows},
Duke Math. J. 119 (2003), no. 3, 465–526.



\bibitem{Higson:Local} N. Higson: \emph{The residue index theorem of Connes and Moscovici}. Surveys in noncommutative geometry, Clay Math. Proc., 6, Amer. Math. Soc., Providence, RI, (2006),
71–126.
 
 \bibitem{Kasparov} G. Kasparov: 
 \emph{Equivariant KK-theory and the Novikov Conjecture}, Invent.\ Math.\ 91 (1988), no.\ 1, pp.\ 147--201.


\bibitem{HR} N. Higson, J. Roe: \emph{Analytic K-homology},
 Oxford Mathematical Monographs. Oxford Science Publications. Oxford University Press, Oxford, (2000). 

\bibitem{LM} M. Lesch, H. Moscovici: \emph{Modular curvature and Morita equivalence}. 
Geom. Funct. Anal. 26 (2016), 818--873. 


\bibitem{Mehler} F.G. Mehler: "\"Uber die Entwicklung einer Function von beliebig vielen Variabeln nach Laplaceschen Functionen höherer Ordnung", Journal für die Reine und Angewandte Mathematik 66.  161-- 176.

 
\bibitem{Roe}. J. Roe: \emph{Elliptic operators, topology and asymptotic methods},
 Pitman Research Notes in Mathematics Series, 179,
  Longman Scientific and Technical, Harlow; copublished in the United States with John Wiley and Sons, Inc., New York, (1988).


\bibitem{Tondeur} Tondeur: foliation book. 



\bibitem{Witten} Gravitional anomolies... 


 \end{thebibliography}
\end{document}